\documentclass[preprint,12pt]{elsarticle}
 
\usepackage[colorlinks=true,linkcolor=blue,citecolor=blue]{hyperref}
\usepackage{amsmath}
\usepackage{amsfonts}
\usepackage{amssymb}
\usepackage{amsthm}
\usepackage{booktabs}
\usepackage{multirow}
\usepackage{float}
\usepackage{graphicx} 
\usepackage{bm}
\usepackage{xcolor}
\usepackage{subcaption}
\usepackage[margin=3cm]{geometry}

\newtheorem{theorem}{Theorem}[section]
\newtheorem{lemma}{Lemma}[section]
\newtheorem{remark}{Remark}[section]

\begin{document}

\begin{frontmatter}

\title{Fast Jacobi Spectral Methods and Closure Approximations for the Homogeneous FENE Model of Complex Fluids}

\affiliation[inst1]{organization={State Key Laboratory of Mathematical Sciences (SKLMS) \& LSEC, Institute of Computational Mathematics and Scientific/Engineering Computing, Academy of Mathematics and Systems Science, Chinese Academy of Sciences},
            city={Beijing},
            postcode={100190}, 
            country={China}}

\affiliation[inst2]{organization={School of Mathematical Sciences, University of Chinese Academy of Sciences},
            city={Beijing},
            postcode={100049}, 
            country={China}}
            
\affiliation[inst3]{organization={School of Mathematical Sciences, Eastern Institute of Technology},
            city={Ningbo},
            postcode={515200}, 
            state={Zhejiang},
            country={China}}

\author[inst2,inst1]{Runkai Feng}
\ead{fengrunkai@amss.ac.cn}

\author[inst3]{Jie Shen}
\ead{jshen@eitech.edu.cn}

\author[inst1,inst2]{Haijun Yu\corref{cor1}}
\ead{hyu@lsec.cc.ac.cn}

\cortext[cor1]{Corresponding author}

\begin{abstract}
The Finitely Extensible Nonlinear Elastic (FENE) dumbbell model is a widely used mathematical model for complex fluids.
Direct simulation of the FENE Fokker--Planck equation is computationally challenging due to high dimensionality and singularity of its potential. In this paper, we develop two fast Jacobi-Spherical Harmonic spectral methods for the spatially homogeneous FENE Fokker--Planck equation. These methods effectively resolve the singularity near the boundary by combining properly designed Jacobi polynomials with a weighted variational formulation. A semi-implicit backward differentiation formula of second-order (BDF2) is employed for time marching, and its energy stability is rigorously proved. The resulting linear algebraic system possesses a sparse structure and can be efficiently solved. Numerical results verify the spectral convergence and efficiency of the direct spectral solvers, establishing them as a reliable tool for generating reference solutions for challenging benchmark problems. Furthermore, to achieve an optimal trade-off between accuracy and efficiency, we compare several closure approximation models, including the industry workhorse Peterlin approximation (FENE-P), the quasi-equilibrium approximation (FENE-QE), and a novel neural network implementation for FENE-QE proposed in this paper (FENE-QE-NN). Numerical experiments in extensional and shear flows demonstrate the superior accuracy and efficiency of the proposed methods compared to traditional approaches.
\end{abstract}


\end{frontmatter}

\section{Introduction}

Complex fluids, composed of polymers or other macromolecules, exhibit rich and intricate dynamical behavior that cannot be captured by classical Newtonian fluid models. The macroscopic behaviors of such non-Newtonian fluids are intimately linked to the microscopic configuration of polymer molecules. The Finitely Extensible Nonlinear Elastic (FENE) dumbbell model (cf., for instance, \cite{Warner72,Herrchen97,Barrett09} and the references therein) is one of the most widely adopted nonlinear kinetic models. Coupled with the Navier--Stokes equations, this model is essential for investigating complex rheological phenomena, such as viscoelastic turbulence and drag modification mechanisms. For instance, recent numerical simulations have utilized viscoelastic models to reveal the existence of a maximum drag enhancement asymptote in turbulent Taylor--Couette flows \cite{Liu24}, highlighting the critical interplay between polymer dynamics and vortical structures. However, accurate simulation of such multiscale systems relies heavily on the efficient solution of the underlying Fokker--Planck equation.

Let the positions of the two ends of the dumbbell be $\bm{r}_1$ and $\bm{r}_2$. The center of mass is denoted by $\bm{x}=(\bm{r}_1+\bm{r}_2)/2$, and the orientation vector of the molecule is $\bm{q}=\bm{r}_1-\bm{r}_2$. In kinetic theory, the configuration distribution function, denoted as $f(t,\bm{x},\bm{q})$, describes the probability density of molecules with orientation $\bm{q}$ at position $\bm{x}$ and time $t$. The microscopic system of complex fluids is governed by the Fokker--Planck equation:
\begin{equation}
    \frac{\partial f}{\partial t}+\boldsymbol{u}\cdot\nabla_{\bm{x}}f+\nabla_{\bm{q}}\cdot(\bm{K}\cdot \bm{q}f)=\frac{1}{\text{De}}\Delta_{\bm{q}}f+\frac{1}{\text{De}}\nabla_{\bm{q}}\cdot\left(f\nabla_{\bm{q}}U\right),\label{eq:FP eq}
\end{equation}
where $\bm{K}=(\nabla_{\bm{x}}\bm{u})^\top$, and $U$ is the FENE elastic potential energy given by:
\begin{equation}
    U=-\frac{b}{2}\ln(1-|\bm{q}|^2).
\end{equation}
The system involves dimensionless parameters $b$ and the Deborah number ($\text{De}$), defined as:
\begin{equation}
    b=\frac{Hq_0^2}{k_BT},\quad \text{De}=\frac{\zeta q_0^2}{2t_0k_BT}.
\end{equation}
Here, $q_0$ is the maximum extensibility of the dumbbell, $k_B$ is the Boltzmann constant, $T$ is the temperature, $H$ is the FENE spring constant, $\zeta$ is the friction coefficient of the polymer molecule, and $t_0$ is the characteristic time of the background flow.

While the Fokker--Planck equation \eqref{eq:FP eq} is linear with respect to $f$ for a given $\bm{K}$, the unbounded nature of the FENE potential $U$ introduces a logarithmic singularity as $|\bm{q}| \to 1$. This singularity poses significant challenges for both rigorous mathematical analysis and stable numerical approximations. Early theoretical works focused on local well-posedness. But, substantial progress has been made recently regarding the global existence, regularity, and rigorous proof of convergence to equilibrium states by exploiting the system's entropy dissipation structure \cite{Blaustein24,Yang22,Masmoudi13}.

From a numerical perspective, kinetic theory is typically approached via two distinct paradigms. The stochastic approach, such as CONNFFESSIT and the Brownian configuration fields (BCF) method (see, e.g., \cite{Hulsen97,Laso93}), relies on Monte Carlo sampling. While the computational cost grows mildly with degrees of freedom, its accuracy is inherently limited by statistical noise ($O(1/\sqrt{N})$). Conversely, the deterministic approach involves solving the governing equations directly. Recent advancements in finite difference methods have achieved high fidelity in simulating viscoelastic turbulence by incorporating distinct criteria to ensure the positive definiteness of the conformation tensor \cite{Liu22}. However, for the microscopic Fokker--Planck equation, spectral methods remain the favored choice due to their exponential convergence rates, which are crucial for resolving the steep gradients of the distribution function near the boundaries \cite{LiuC21,LiuC24,Shen12}. Previous work \cite{Shen12} has proven the well-posedness of the weak formulation and designed a fast spectral method for the homogeneous two-dimensional (0+2) case. Additionally, the convection term in \eqref{eq:FP eq} can be treated by a classical finite difference scheme or a spectral element method; see, e.g., \cite{Barrett09,Chauviere04b}.

Due to the high dimensionality, directly solving the Fokker--Planck equation is computationally expensive, especially in 0+3 dimensions. To address this, researchers often resort to closure approximations~\cite{Du05}, with the FENE-P model \cite{Bird87} being the most common. Of particular interest is the FENE-QE (Quasi-Equilibrium) model \cite{Ilg02}, which offers significantly higher accuracy. However, despite dimension reduction, models like FENE-QE remain computationally intensive due to the need to evaluate complex integrals at every time step. To mitigate this, techniques such as piecewise linear approximation (PLA) \cite{Wang08} have been proposed to improve efficiency.

Recently, data-driven approaches, particularly neural networks, have shown great potential in approximating complex constitutive laws. In the context of non-Newtonian fluid mechanics, the closure problem, mapping the microstructural state to the macroscopic stress, represents a high-dimensional function approximation task well-suited for deep learning. Classical analytical closures, such as the Peterlin approximation, often sacrifice accuracy for computational efficiency. In contrast, neural networks, leveraging their universal approximation capability, can learn high-fidelity constitutive manifolds from data generated by direct kinetic solvers. Several studies have successfully applied NNs to learn stress-strain relationships or to construct closure terms for Reynolds stress in turbulence modeling~\cite{Brunton20,Duraisamy19,McConkey22}.

In this paper, we develop two fast Jacobi-Spherical Harmonic spectral-Galerkin methods for the 0+3 dimensional Fokker--Planck equation. We strictly prove the stability of the semi-implicit backward differentiation formula of second-order (BDF2) time discretization scheme and verify its convergence through rigorous numerical benchmarks. Furthermore, we propose a neural network approach (FENE-QE-NN) to approximate the FENE-QE model, leveraging the high-fidelity data generated by our spectral solver. Numerical experiments in extensional and shear flows demonstrate the accuracy and efficiency of these methods, followed by appendices containing some calculations on the matrix of the spectral-Galerkin method.

\section{The FENE Fokker--Planck equation and its weighted weak form}

Consider a homogeneous system characterized by a three-dimensional configuration space defined on the unit ball $\Omega=\{\bm{q} \in \mathbb{R}^3: |\bm{q}| \le 1\}$. The configuration distribution function is denoted by $f(\bm{q},t)$. The macroscopic velocity gradient tensor $\bm{K}=(\nabla_{\bm{x}}\boldsymbol{u})^\top$ is given by
\begin{equation*}
\bm{K}=
\begin{pmatrix}
k_{11} & k_{12} & k_{13} \\
k_{21} & k_{22} & k_{23} \\
k_{31} & k_{32} & k_{33}
\end{pmatrix},
\quad \text{with } \mathrm{tr}(\bm{K}) = k_{11}+k_{22}+k_{33}=0,
\end{equation*}
which implies that the flow is incompressible. We decompose $\bm{K}$ into its symmetric and antisymmetric parts: the rate-of-strain tensor $\bm{D}=(\bm{K}+\bm{K}^\top)/2$ and the vorticity tensor $\bm{W}=(\bm{K}-\bm{K}^\top)/2$.

The Fokker--Planck equation governing the evolution of $f$ for the FENE model is given by
\begin{equation}
    \frac{\partial f}{\partial t}+\nabla_{\bm{q}}\cdot(\bm{K}\cdot \bm{q}f)=\frac{1}{\mathrm{De}}\Delta_{\bm{q}}f+\frac{b}{\mathrm{De}}\nabla_{\bm{q}}\cdot\left(f\frac{\bm{q}}{1-|\bm{q}|^2}\right),\label{eq:FP equation}
\end{equation}
where $\mathrm{De}$ denotes the Deborah number. The associated free energy functional is defined as
\begin{equation}
    A[f]=\int_{\Omega}(f\ln f-f)\,\mathrm{d}\bm{q}+\int_{\Omega}\left\{U(\bm{q})f - f\frac{\mathrm{De}}{2}\bm{D}:\bm{q}\bm{q}\right\}\,\mathrm{d}\bm{q},
\end{equation}
where $U(\bm{q})$ represents the FENE spring potential.

It is well-known that for irrotational flows ($\bm{W}=\bm{0}$), the solution to \eqref{eq:FP equation} converges to a steady state that minimizes the free energy \cite{Warner72}:
\begin{equation}
    f_{eq}^{\bm{D}}(\bm{q})=\frac{1}{Z_{\bm{D}}}(1-|\bm{q}|^2)^{b/2}
    \exp\left(\frac{\mathrm{De}}{2}\bm{D}:\bm{q}\bm{q}\right),\label{eq:f_eq^D}
\end{equation}
where $Z_{\bm{D}}$ is the normalization constant. In the absence of flow ($\bm{K}=\bm{0}$), this reduces to the equilibrium distribution:
\begin{equation}
    f_{eq}(\bm{q})=C_{eq}(1-|\bm{q}|^2)^{b/2},\quad C_{eq}=\frac{\Gamma(b/2+5/2)}{2\pi^{3/2}\Gamma(b/2+1)}.\label{eq:f_eq}
\end{equation}

Multiplying \eqref{eq:FP equation} by a test function $g \in C_0^{\infty}(\Omega)$ and applying integration by parts, we obtain the standard weak formulation:
\begin{equation}
    \left(\frac{\partial f}{\partial t},g\right)+\frac{1}{\mathrm{De}}\left(\nabla f,\nabla g\right)+\frac{b}{\mathrm{De}}\left(f,\frac{\bm{q}\cdot \nabla g}{1-|\bm{q}|^2}\right)=(f,(\bm{K}\cdot\bm{q})\cdot \nabla g).\label{eq:weak formulation1}
\end{equation}
To handle the singularity at the boundary $|\bm{q}|=1$, we introduce the transformation
\begin{equation*}
f(\bm{q},t)=(1-|\bm{q}|^2)^s h(\bm{q},t),\quad 1<s\leq b/2.
\end{equation*}
Substituting this into \eqref{eq:weak formulation1}, we derive the weighted weak formulation for the new unknown $h\in H^1_s(\Omega)$:
\begin{equation}
    \left(\frac{\partial h}{\partial t},g\right)_s+\frac{1}{\mathrm{De}}\left(\nabla h,\nabla g\right)_s+\frac{b-2s}{\mathrm{De}}\left(h,\frac{\bm{q}\cdot \nabla g}{1-|\bm{q}|^2}\right)_s=(h,(\boldsymbol{K}\cdot\bm{q})\cdot \nabla g)_s,\label{eq:weak formulation2}
\end{equation}
where the weighted inner product is defined as $(h,g)_s=\int_{\Omega}hg(1-|\bm{q}|^2)^s\,\mathrm{d}\bm{q}$, and $H_s^1(\Omega)$ is the weighted Sobolev space (see \cite{Kufner85}). The well-posedness and energy stability of this formulation have been established in \cite{Shen12}.

In spherical coordinates ($q_1=r\sin\theta \cos\varphi, q_2=r\sin\theta \sin\varphi, q_3=r\cos\theta$), the domain becomes $\Sigma=\{(r,\theta,\varphi): r\in[0,1], \theta\in[0,\pi], \varphi\in[0,2\pi)\}$. The inner product transforms to
\begin{equation}
    \left(h,g\right)_s=\int_{\Sigma}hg(1-r^2)^s r^2\sin\theta \,\mathrm{d}r\mathrm{d}\theta \mathrm{d}\varphi.
\end{equation}
Consequently, the weak formulation \eqref{eq:weak formulation2} is rewritten as
\begin{equation}
\begin{aligned}
    \left(\frac{\partial h}{\partial t},g\right)_s &+\frac{1}{\mathrm{De}}\left(h_rg_r+\frac{h_{\theta}g_{\theta}}{r^2}+
    \frac{h_{\varphi}g_{\varphi}}{r^2\sin^2\theta}\right)_s+\frac{b-2s}{\mathrm{De}}\left(h,\frac{rg_r}{1-r^2}\right)_s\\
    &= k_{ij}\left(h, r g_r \hat{r}_i \hat{r}_j + g_{\theta}\hat{\theta}_i \hat{r}_j + \frac{g_{\varphi}\hat{\varphi}_i \hat{r}_j}{\sin\theta}\right)_s, \label{eq:weak formulation3}
\end{aligned}
\end{equation}
where the Einstein summation convention is assumed for indices $i,j$. Here, $\hat{r}_i, \hat{\theta}_i, \hat{\varphi}_i$ denote the components of the unit vectors $\hat{\bm{r}}, \hat{\bm{\theta}}, \hat{\bm{\varphi}}$ defined below.

To facilitate the spectral approximation, we apply the coordinate mapping
\begin{equation}
    r^2=\frac{1+p}{2}, \quad p\in [-1,1].
\end{equation}
Let $\Sigma'=[-1,1]\times[0,\pi]\times[0,2\pi)$. This transformation, utilized in \cite{Chauviere04a,Shen12}, offers two primary advantages: (i) it redistributes Gauss quadrature points away from the pole, mitigating clustering issues; and (ii) it reduces the polynomial degree in the weight function by half. Applying this mapping (noting that $g_r=4r g_p$) leads to the final weak formulation:
\begin{equation}
    \begin{aligned}
        \left(\frac{\partial h}{\partial t},g\right)_{s,\frac{1}{2}} &+ \frac{1}{\mathrm{De}}\left(16\frac{1+p}{2}h_pg_p+\frac{2}{1+p}\left(h_{\theta}g_{\theta}+\frac{h_{\varphi}g_{\varphi}}{\sin^2\theta}\right)\right)_{s,\frac{1}{2}} \\
        &+ \frac{b-2s}{\mathrm{De}}\left(h,4\frac{1+p}{1-p}g_p\right)_{s,\frac{1}{2}} \\
        &= k_{ij}\left(h, 4\frac{1+p}{2}g_p \hat{r}_i \hat{r}_j + g_{\theta}\hat{\theta}_i \hat{r}_j + \frac{g_{\varphi}\hat{\varphi}_i \hat{r}_j}{\sin\theta}\right)_{s,\frac{1}{2}}, \label{eq:weak formulation4}
    \end{aligned}
\end{equation}
where the generalized inner product is defined as
\begin{equation}
    (h,g)_{\alpha,\beta}=\int_{-1}^1\int_0^{\pi}\int_0^{2\pi}hg(1-p)^{\alpha}(1+p)^{\beta}\sin\theta\,\mathrm{d}p\mathrm{d}\theta\mathrm{d}\varphi.
\end{equation}
The unit vectors involved in the convection term are given by
\begin{equation*}
\begin{aligned}
    \hat{\bm{r}}&=(\sin\theta \cos\varphi,\sin\theta \sin\varphi,\cos\theta)^\top,\\
    \hat{\bm{\theta}}&=(\cos\theta \cos\varphi,\cos\theta \sin\varphi,-\sin\theta)^\top,\\
    \hat{\bm{\varphi}}&=(-\sin\varphi,\cos\varphi,0)^\top.
\end{aligned}
\end{equation*}
  
\section{Time discretization and stability analysis}

We employ a second-order Backward Differentiation Formula (BDF2) for time discretization. Let $\Delta t > 0$ be the time step size, and let $h^n$ denote the numerical approximation at time $t_n = n\Delta t$. To decouple the system and maintain efficiency, we apply a second-order explicit linear extrapolation to the convection term $(h,(\boldsymbol{K}\cdot\boldsymbol{q})\cdot \nabla g)_s$. The semi-discrete variational scheme reads as follows: for any test function $g \in H_s^1(\Omega)$, find $h^{n+1}\in H_s^1(\Omega)$ such that
\begin{equation} 
\label{eq:BDF2}
\begin{aligned}
    & \frac{1}{2\Delta t}\left(3h^{n+1}-4h^{n}+h^{n-1}, g\right)_s + \frac{1}{\text{De}}\left(\nabla h^{n+1},\nabla g\right)_s \\
    &+\frac{b-2s}{\text{De}}\left(h^{n+1},\frac{\bm{q}\cdot \nabla g}{1-|\bm{q}|^2}\right)_s 
    = (2h^{n}-h^{n-1}, (\bm{K}\cdot\bm{q})\cdot\nabla g)_s.
\end{aligned}
\end{equation}

We first address the solvability of the semi-discrete scheme \eqref{eq:BDF2}. At each time step $t_{n+1}$, given $h^n$ and $h^{n-1}$, the problem reduces to a linear elliptic problem. Let $a(\cdot, \cdot): H_s^1(\Omega) \times H_s^1(\Omega) \to \mathbb{R}$ be the bilinear form defined by the left-hand side of \eqref{eq:BDF2}:
\begin{equation}
    a(u, v) = \frac{3}{2\Delta t}(u, v)_s + \frac{1}{\text{De}}(\nabla u, \nabla v)_s + \frac{b-2s}{\text{De}}\left(u, \frac{\bm{q}\cdot \nabla v}{1-|\bm{q}|^2}\right)_s.
\end{equation}
The linear functional on the right-hand side is denoted by $L(v)$. To prove the coercivity of $a(\cdot, \cdot)$, we recall a crucial identity derived in \cite{Shen12}.

\begin{lemma}
\label{lem:Shen_identity}
For any $h \in H_s^1(\Omega)$ and $s > 1$, the following identity holds:
\begin{equation} \label{eq:identity}
    \left(h, \frac{\bm{q}\cdot \nabla h}{1-|\bm{q}|^2}\right)_s = (s-1)\|\bm{q}h\|_{s-2}^2 - \frac{3}{2}\|h\|_{s-1}^2.
\end{equation}
Furthermore, by analyzing the quadratic form associated with the weighted norms, there exists a threshold $\gamma(s)$ such that for any $\gamma > \gamma(s)$, the following inequality holds:
\begin{equation} \label{eq:positivity}
    \gamma \|h\|_s^2 + \left(h, \frac{\bm{q}\cdot \nabla h}{1-|\bm{q}|^2}\right)_s \ge \beta(\gamma)\|h\|_{s-2}^2,
\end{equation}
where 
\begin{equation}
    \gamma(s) = \frac{3}{4}\left(1 + \frac{s-1}{3} + \frac{3}{4(s-1)}\right), \quad \beta(\gamma) = (s-1)\left(1 - \frac{\gamma(s)}{\gamma}\right).
\end{equation}
\end{lemma}

\begin{theorem}[Existence and Uniqueness]
Assume $1 < s \le b/2$. If $s < b/2$, we further assume the time step $\Delta t$ satisfies the condition $\frac{3}{2\Delta t} > \frac{b-2s}{\mathrm{De}}\gamma(s)$. Then, for any given $h^n, h^{n-1} \in L^2_s(\Omega)$, the variational problem \eqref{eq:BDF2} admits a unique solution $h^{n+1} \in H_s^1(\Omega)$.
\end{theorem}

\begin{proof}
The continuity of the bilinear form $a(u,v)$ is straightforward. For coercivity, we consider $a(u,u)$. 
If $s = b/2$, the third term in $a(u,u)$ vanishes, and coercivity follows immediately from the positivity of the time-derivative and diffusion terms. 
If $s < b/2$, we utilize Lemma \ref{lem:Shen_identity}. By splitting the coefficient of the time derivative term as $\frac{3}{2\Delta t} = \lambda + \frac{b-2s}{\text{De}}\gamma$ with $\lambda > 0$ and $\gamma > \gamma(s)$, we obtain:
\begin{equation}
    a(u, u) = \lambda \|u\|_s^2 + \frac{1}{\text{De}}\|\nabla u\|_s^2 + \frac{b-2s}{\text{De}}\left[ \gamma \|u\|_s^2 + \left(u, \frac{\bm{q}\cdot \nabla u}{1-|\bm{q}|^2}\right)_s \right].
\end{equation}
Using \eqref{eq:positivity}, the term in the brackets is non-negative. Thus, $a(u,u) \ge C \|\nabla u\|_s^2$, which implies coercivity in $H_s^1(\Omega)$. The existence and uniqueness follow from the Lax-Milgram theorem.
\end{proof}

We now establish the energy stability of the proposed BDF2 scheme.

\begin{theorem}[Stability]
\label{thm:stability}
Let $h^0 \in L^2_s(\Omega)$ be the initial data. Assume $1 < s \le b/2$. 
\begin{enumerate}
    \item \textbf{Case I ($s = b/2$):} The scheme \eqref{eq:BDF2} is unconditionally stable.
    \item \textbf{Case II ($s < b/2$):} The scheme is stable provided that the time step satisfies $\Delta t < \frac{3\mathrm{De}}{2(b-2s)\gamma(s)}$.
\end{enumerate}
In both cases, the solution satisfies the following energy estimate for any $N \ge 1$:
\begin{equation}
\begin{aligned}
    \|h^N\|_s^2 &+ \|2h^N - h^{N-1}\|_s^2 + \sum_{n=1}^{N-1} \|h^{n+1} - 2h^n + h^{n-1}\|_s^2 \\
    &+ \frac{2\Delta t}{\text{De}} \sum_{n=1}^{N-1} \|\nabla h^{n+1}\|_s^2 + 2\Delta t \sum_{n=1}^{N-1} \mathcal{P}(h^{n+1}) \le C(T, h^0, h^1),
\end{aligned}
\end{equation}
where $\mathcal{P}(h) = \beta(\gamma)\|h\|_{s-2}^2 \ge 0$ is the dissipation contribution from the FENE potential (with $\beta(\gamma)=0$ when $s=b/2$).
\end{theorem}

\begin{proof}
We choose the test function $g = h^{n+1}$ in \eqref{eq:BDF2}. 
Using the identity $2(3a-4b+c, a) = |a|^2 + |2a-b|^2 - |b|^2 - |2b-c|^2 + |a-2b+c|^2$, the first term becomes:
\begin{equation}
    \frac{1}{2\Delta t}\left(3h^{n+1}-4h^{n}+h^{n-1}, h^{n+1}\right)_s = \frac{1}{4\Delta t}\left( \mathcal{E}^{n+1} - \mathcal{E}^n + \|h^{n+1}-2h^n+h^{n-1}\|_s^2 \right),
\end{equation}
where $\mathcal{E}^{n+1} = \|h^{n+1}\|_s^2 + \|2h^{n+1}-h^n\|_s^2$.

For the convection term on the right-hand side, we apply the Cauchy-Schwarz inequality, Young's inequality, and the bound $\|\boldsymbol{K}\cdot\boldsymbol{q}\|_{L^\infty}^2 \le |\boldsymbol{K}|_\infty^2$:
\begin{equation}
\begin{aligned}
    \left( (2h^n - h^{n-1}) (\bm{K}\cdot \bm{q}), \nabla h^{n+1} \right)_s 
    &\le |\boldsymbol{K}|_\infty \|2h^n - h^{n-1}\|_s \|\nabla h^{n+1}\|_s \\
    &\le \frac{1}{2\text{De}} \|\nabla h^{n+1}\|_s^2 + \frac{\text{De}|\boldsymbol{K}|_\infty^2}{2} \|2h^n - h^{n-1}\|_s^2.
\end{aligned}
\end{equation}

Now we handle the FENE potential term.
\textbf{Case I ($s=b/2$):} The term involving $\frac{b-2s}{\text{De}}$ vanishes. We simply keep $\frac{3}{4\Delta t}\|h^{n+1}\|_s^2$ on the LHS.
\textbf{Case II ($s<b/2$):} We borrow a portion of the mass term from the time discretization to stabilize the potential term. Specifically, we set $\frac{3}{4\Delta t} = \frac{1}{4\Delta t} + \frac{b-2s}{\text{De}}\gamma$ (implying $\gamma = \frac{\text{De}}{2\Delta t(b-2s)}$). The stability condition $\Delta t < \frac{3\text{De}}{2(b-2s)\gamma(s)}$ ensures that $\gamma > \gamma(s)$. By Lemma \ref{lem:Shen_identity}, we have:
\begin{equation}
    \frac{b-2s}{\text{De}}\gamma \|h^{n+1}\|_s^2 + \frac{b-2s}{\text{De}}\left(h^{n+1}, \frac{\bm{q}\cdot \nabla h^{n+1}}{1-|\bm{q}|^2}\right)_s \ge  \frac{b-2s}{\text{De}}\beta(\gamma)\|h^{n+1}\|_{s-2}^2.
\end{equation}

Combining these estimates, we arrive at:
\begin{equation}
\begin{aligned}
    \frac{1}{4\Delta t} &(\mathcal{E}^{n+1} - \mathcal{E}^n) + \frac{1}{4\Delta t}\|h^{n+1}-2h^n+h^{n-1}\|_s^2 \\
    &+ \frac{1}{2\text{De}}\|\nabla h^{n+1}\|_s^2 +  \frac{b-2s}{\text{De}}\beta(\gamma)\|h^{n+1}\|_{s-2}^2 \\
    &\le \frac{\text{De}|\boldsymbol{K}|_\infty^2}{2} \|2h^n - h^{n-1}\|_s^2.
\end{aligned}
\end{equation}
Multiplying by $4\Delta t$, summing from $n=1$ to $N-1$, and noting that $\|2h^n - h^{n-1}\|_s^2 \le \mathcal{E}^n$, we can apply the discrete Gronwall inequality. Since the term $\beta(\gamma)\|h^{n+1}\|_{s-2}^2$ is non-negative, it is retained on the left-hand side to improve the stability bound. This completes the proof.
\end{proof}

\begin{remark}[Asymptotic behavior and stability at low $\mathrm{De}$ number]
The conditional stability criterion established in Theorem \ref{thm:stability} requires the time step to satisfy $\Delta t \lesssim \frac{\mathrm{De}}{b-2s}$. In the small Deborah number limit ($\mathrm{De} \to 0$), a fixed weight index $s < b/2$ would impose a prohibitively restrictive time-step constraint ($\Delta t \to 0$), leading to severe numerical stiffness. To mitigate this, one must enforce the scaling $b-2s = O(\mathrm{De})$, which implies $s$ should approach $b/2$. 
This choice is not merely a numerical artifact but is physically justified. Asymptotic analyses of the Fokker--Planck equation (see, e.g., \cite{Jourdain04,Zhang06}) indicate that the probability density function $f$ rapidly relaxes to the equilibrium profile near the boundary, behaving as $f \sim (1-|\bm{q}|^2)^{b/2}$. Consequently, setting $s \approx b/2$ ensures that the transformed variable $h = f/(1-|\bm{q}|^2)^s$ remains regular and bounded, thereby preserving both numerical stability and physical fidelity in the small $\mathrm{De}$ regime.
\end{remark}

\section{Spectral discretization in configuration space}

\subsection{Spherical harmonic expansion and symmetry}
We approximate the dependence on the angular coordinates $(\theta, \varphi)$ using a truncated expansion in real spherical harmonics:
\begin{equation}          
    h(p,\theta,\varphi)=\sum_{l=0}^{L}\sum_{v=0,1}\sum_{m=v}^l\psi_{lm}^v(p)Y_{lm}^v(\theta,\varphi), \label{eq:expansion_h}
\end{equation}
where the real spherical harmonics $Y_{lm}^v$ are defined as
\begin{equation}
    Y_{lm}^v(\theta,\varphi)=C_l^m P_l^m(\cos\theta)e_m^v(\varphi),\quad \text{with } C_l^m=\sqrt{\frac{(2l+1)(l-m)!}{2(l+m)!}}.
\end{equation}
Here, $P_l^m$ denotes the associated Legendre polynomials, and the azimuthal basis functions are given by
\begin{equation}
    e_m^v(\varphi)=\begin{cases}
        \frac{1}{\sqrt{2\pi}}, & v=0, m=0, \\
        \frac{1}{\sqrt{\pi}}\cos(m\varphi), & v=0, m>0, \\
        \frac{1}{\sqrt{\pi}}\sin(m\varphi), & v=1, m>0.
    \end{cases}
\end{equation}

The FENE dumbbell model assumes that the two ends of the polymer molecule are indistinguishable, implying that the probability distribution function must satisfy the head-tail symmetry condition:
\begin{equation}
    h(p,\theta,\varphi) = h(p,\pi-\theta,\pi+\varphi).\label{eq:symmetry}
\end{equation}
Recalling the parity property of spherical harmonics, $Y_{lm}^v(\pi-\theta,\pi+\varphi)=(-1)^lY_{lm}^v(\theta,\varphi)$, the difference between the function and its centrally symmetric counterpart is
\begin{equation*}
    h(p,\theta,\varphi)-h(p,\pi-\theta,\pi+\varphi) = \sum_{l,v,m}\psi_{lm}^v(p)\left[1-(-1)^l\right]Y_{lm}^v(\theta,\varphi).
\end{equation*}
For this difference to vanish, the coefficients must satisfy $\psi_{lm}^v(p) = 0$ for all odd $l$. Consequently, the summation in \eqref{eq:expansion_h} is restricted to even degrees $l$.

Regarding the radial direction, regularity at the pole (corresponding to $r=0$ or $p=-1$) imposes constraints on the behavior of $\psi_{lm}^v(p)$. To ensure $C^1$ continuity of the distribution function in $\bm{q}$, the coefficient functions must behave as $\psi_{lm}^v(p)=O((1+p)^{l/2})$ as $p\to -1$. For the mapped coordinate $p$, we consider two different basis choices to address this regularity.

\subsection{Fast Jacobi Galerkin method with $C^1$ continuity (JG1)}
In the first approach, denoted as JG1, we select the basis functions in the $p$-direction to satisfy the essential pole conditions. The radial basis functions are defined as
\begin{equation}
    \phi_{ln}(p)=\begin{cases}
        J_n^{s-2,\frac{1}{2}}(p), & l=0, \\
        (1+p)J_n^{s-2,\frac{3}{2}}(p), & l>0,
    \end{cases}
\end{equation}
where $J_n^{\alpha,\beta}(p)$ represents the Jacobi polynomial of degree $n$ with indices $\alpha, \beta$. The full approximation space is spanned by
\begin{equation}
    h_N(p,\theta,\varphi)=\sum_{l=0, \text{even}}^{L}\sum_{v=0,1}\sum_{m=v}^{l}\sum_{n=0}^{N} b_{vlmn}\phi_{ln}(p)Y_{lm}^v(\theta,\varphi).
\end{equation}
Substituting this expansion into the weak formulation \eqref{eq:weak formulation4} and choosing the test function $g=\phi_{l'n'}(p)Y_{l'm'}^{v'}(\theta,\varphi)$, we obtain the following linear algebraic system:
\begin{equation}
    \mathcal{A} \frac{\mathrm{d}\mathbf{b}}{\mathrm{d}t} + \frac{1}{\mathrm{De}}\mathcal{B}\mathbf{b} + \frac{b-2s}{\mathrm{De}}\mathcal{C}\mathbf{b} = k_{ij}\mathcal{D}_{ij}\mathbf{b}.
\end{equation}
Element-wise, the equation reads
\begin{equation}
    A_{vlmn}^{v'l'm'n'}\dot{b}_{vlmn} + \frac{1}{\mathrm{De}}B_{vlmn}^{v'l'm'n'}b_{vlmn} + \frac{b-2s}{\mathrm{De}}C_{vlmn}^{v'l'm'n'}b_{vlmn} = k_{ij}(D_{ij})_{vlmn}^{v'l'm'n'}b_{vlmn},
\end{equation}
where the mass, stiffness, and potential matrices are decoupled due to the $L^2$-orthogonality of spherical harmonics:
\begin{align*}
    A_{vlmn}^{v'l'm'n'} &= \delta_{vv'}\delta_{ll'}\delta_{mm'} (O_l)_n^{n'}, \\
    B_{vlmn}^{v'l'm'n'} &= \delta_{vv'}\delta_{ll'}\delta_{mm'} \left(8(P_l)_n^{n'} + 2l(l+1)(Q_l)_n^{n'}\right), \\
    C_{vlmn}^{v'l'm'n'} &= \delta_{vv'}\delta_{ll'}\delta_{mm'} \left(4(R_l)_n^{n'}\right), \\
    (D_{ij})_{vlmn}^{v'l'm'n'} &= (U_{ij})_{vlm}^{v'l'm'} 2(S_{ll'})_n^{n'} + (V_{ij}+W_{ij})_{vlm}^{v'l'm'} (O_{ll'})_n^{n'}.
\end{align*}
The radial matrices are defined by the inner products in the mapped domain:
\begin{align*}
    (O_l)_n^{n'} &= (\phi_{ln},\phi_{ln'})_{s,\frac{1}{2}}, & (P_l)_n^{n'} &= ((1+p)\phi_{ln}',\phi_{ln'}')_{s,\frac{1}{2}}, \\
    (Q_l)_n^{n'} &= \left(\frac{1}{1+p}\phi_{ln},\phi_{ln'}\right)_{s,\frac{1}{2}}, & (R_l)_n^{n'} &= \left(\frac{1+p}{1-p}\phi_{ln},\phi_{ln'}'\right)_{s,\frac{1}{2}}, \\
    (S_{ll'})_n^{n'} &= ((1+p)\phi_{ln},\phi_{l'n'}')_{s,\frac{1}{2}}, & (O_{ll'})_n^{n'} &= (\phi_{ln},\phi_{l'n'})_{s,\frac{1}{2}}.
\end{align*}
The angular interaction matrices are given by:
\begin{align*}
    (U_{ij})_{vlm}^{v'l'm'} &= (Y_{lm}^v, Y_{l'm'}^{v'} \hat{r}_i \hat{r}_j)_{\sin\theta}, \\
    (V_{ij})_{vlm}^{v'l'm'} &= (Y_{lm}^v, \partial_{\theta}Y_{l'm'}^{v'} \hat{\theta}_i \hat{r}_j)_{\sin\theta}, \\
    (W_{ij})_{vlm}^{v'l'm'} &= \left(Y_{lm}^v, \partial_{\varphi}Y_{l'm'}^{v'} \frac{\hat{\varphi}_i \hat{r}_j}{\sin\theta}\right)_{\sin\theta},
\end{align*}
where $(\cdot,\cdot)_{\sin\theta}$ denotes the standard $L^2$ inner product on the unit sphere.

A key feature of this method is the sparsity of the radial matrices. Specifically, $(O_l), (P_l), (Q_l)$, and $(R_l)$ are banded matrices with bandwidths of 7, 5, 5, and 5, respectively (for $l=0$, the contribution from $Q_l$ vanishes). The detailed structures of these matrices are provided in the appendix of \cite{Shen12}. While the convection matrix $(D_{ij})$ is dense in the radial direction, its angular blocks are sparse, which ensures computational efficiency (see Appendix A for details).

\subsection{Fast Jacobi Galerkin method with $C^{\infty}$ continuity (JGinf)}
To achieve higher spectral accuracy, particularly near the pole, we propose the JGinf method which enforces a stronger regularity condition consistent with the analytic behavior of the solution. The radial basis functions are chosen as:
\begin{equation}
    \psi_{ln}(p)=\begin{cases}
        J_n^{s-2,\frac{1}{2}}(p), & l=0, \\
        (1+p)^l J_n^{s-2,\frac{1}{2}+2l-1}(p), & l>0.
    \end{cases}
\end{equation}
The approximation space is defined similarly to JG1, but with an $l$-dependent polynomial degree constraint to avoid redundancy and ensure a uniform total degree:
\begin{equation}
    h_N(p,\theta,\varphi)=\sum_{l=0, \text{even}}^{L}\sum_{v=0,1}\sum_{m=v}^{l}\sum_{n=0}^{N-l} b_{vlmn}\psi_{ln}(p)Y_{lm}^v(\theta,\varphi).
\end{equation}
The resulting linear system shares the same structure as the JG1 system. Importantly, the radial matrices $(O_l, P_l, \dots)$ computed using the new basis $\psi_{ln}$ retain their sparsity properties, possessing bandwidths of 7, 5, 5, and 5, respectively. This choice of basis ensures that the numerical solution correctly captures the asymptotic behavior of the distribution function near the coordinate singularity $r=0$.

\section{Moment closure approximation models}

Directly solving the high-dimensional Fokker--Planck equation is often computationally prohibitive for practical flow simulations. To mitigate this curse of dimensionality while retaining essential microscopic information, a common strategy is to derive evolution equations for the lower-order moments of the configuration distribution function. The most prevalent approach focuses on the second-order moment, known as the conformation tensor.

\subsection{Derivation of the moment equation}
We define the conformation tensor $\bm{C}$ as the second-order moment of the distribution function $f(\bm{q}, t)$:
\begin{equation}
    \bm{C}(t) = \langle \bm{q}\bm{q} \rangle = \int_{\Omega} \bm{q}\bm{q} f(\bm{q}, t) \,\mathrm{d}\bm{q}.
\end{equation}
To derive the evolution equation for $\bm{C}$, we multiply Equation \eqref{eq:FP equation} by the dyadic product $\bm{q}\bm{q}$ and integrate over the configuration space $\Omega$:
\begin{equation}
    \frac{\mathrm{d}\bm{C}}{\mathrm{d}t} = \int_{\Omega}\bm{q}\bm{q}\left[-\nabla_{\bm{q}}\cdot(\bm{K}\cdot \bm{q}f)+\frac{1}{\mathrm{De}}\Delta_{\bm{q}}f+\frac{b}{\mathrm{De}}\nabla_{\bm{q}}\cdot\left(f\frac{\bm{q}}{1-|\bm{q}|^2}\right)\right]\,\mathrm{d}\bm{q}.\label{eq:C1}
\end{equation}
By applying integration by parts and standard tensor identities, the terms on the right-hand side are evaluated as follows:
\begin{align*}
    \text{Convection:} \quad & \int_{\Omega} \bm{q}\bm{q} \left[ -\nabla_{\bm{q}}\cdot(\bm{K}\cdot \bm{q}f) \right]\,\mathrm{d}\bm{q} = \bm{K} \cdot \bm{C} + \bm{C} \cdot \bm{K}^\top, \\
    \text{Diffusion:} \quad & \frac{1}{\mathrm{De}} \int_{\Omega} \bm{q}\bm{q} (\Delta_{\bm{q}}f)\,\mathrm{d}\bm{q} = \frac{1}{\mathrm{De}} \int_{\Omega} f (\Delta_{\bm{q}}(\bm{q}\bm{q}))\,\mathrm{d}\bm{q} = \frac{2}{\mathrm{De}}\bm{I}, \\
    \text{Spring force:} \quad & \frac{b}{\mathrm{De}} \int_{\Omega} \bm{q}\bm{q} \left[ \nabla_{\bm{q}}\cdot\left(f\frac{\bm{q}}{1-|\bm{q}|^2}\right) \right]\,\mathrm{d}\bm{q} = -\frac{2b}{\mathrm{De}} \left\langle \frac{\bm{q}\bm{q}}{1-|\bm{q}|^2} \right\rangle.
\end{align*}
Combining these results yields the exact, but unclosed, evolution equation for $\bm{C}$:
\begin{equation}
    \frac{\mathrm{d}\bm{C}}{\mathrm{d}t}=\bm{K} \cdot \bm{C} + \bm{C} \cdot \bm{K}^\top + \frac{2}{\mathrm{De}}\bm{I} - \frac{2b}{\mathrm{De}} \left\langle \frac{\bm{q}\bm{q}}{1-|\bm{q}|^2} \right\rangle.\label{eq:C_unclosed}
\end{equation}
The system is unclosed because the nonlinear spring term involves higher-order moments (e.g., via the Taylor expansion $\left\langle \frac{\bm{q}\bm{q}}{1-|\bm{q}|^2} \right\rangle = \sum_{k=0}^{\infty} \langle \bm{q}\bm{q} |\bm{q}|^{2k} \rangle$). To create a self-contained system, this term must be approximated as a function of $\bm{C}$, a process known as closure approximation.

\subsection{The FENE-P model}
The FENE-P model employs the widely used Peterlin approximation \cite{Bird87}, a pre-averaging technique that replaces the denominator in the spring term with its ensemble average:
\begin{equation}
    \frac{1}{1-|\bm{q}|^2} \approx \frac{1}{1-\langle |\bm{q}|^2 \rangle} = \frac{1}{1-\mathrm{tr}(\bm{C})}.
\end{equation}
Substituting this into \eqref{eq:C_unclosed} yields the closed evolution equation for the FENE-P model:
\begin{equation}
    \frac{\mathrm{d}\bm{C}}{\mathrm{d}t}=\bm{K} \cdot \bm{C} + \bm{C} \cdot \bm{K}^\top + \frac{2}{\mathrm{De}}\bm{I} - \frac{b}{\mathrm{De}} \frac{\bm{C}}{1-\mathrm{tr}(\bm{C})}.\label{eq:fenep}
\end{equation}
Under this approximation, it can be shown that if the initial distribution is Gaussian, the PDF remains Gaussian for all time \cite{Ottinger87}, taking the form $f_{\mathrm{P}}(\bm{q}) \propto \exp(-\frac{1}{2}\bm{C}^{-1}:\bm{q}\bm{q})$. The FENE-P model is computationally inexpensive and has been proven to be thermodynamically consistent with a modified free energy functional \cite{Hu07}.

Although the Peterlin approximation cannot represent strongly non-Gaussian configuration statistics, the FENE-P closure is theoretically well-motivated in several important regimes. In much of the rheology literature, the flow strength is expressed in terms of the Weissenberg number $\mathrm{Wi}\sim \mathrm{De}\,|\bm{K}|_{\infty}$, so that small (resp. large) $\mathrm{De}\,|\bm{K}|_{\infty}$ corresponds to weak (resp. strong) flow-induced polymer deformation.

From the modeling viewpoint, FENE-P recovers the Hookean-dumbbell/Oldroyd-B limit in the large-$b$ regime, in which the finite-extensibility effects become weak, and the spring behaves effectively linear on the typical configuration scale \cite{Bird87}. Moreover, when the polymer remains far from the extensibility boundary so that $\mathrm{tr}(\bm{C})\ll 1$, the Peterlin factor admits the elementary expansion
\begin{equation}
    \frac{1}{1-\mathrm{tr}(\bm{C})}
    = 1 + \mathrm{tr}(\bm{C}) + O\!\left(\mathrm{tr}(\bm{C})^2\right),
\end{equation}
which shows that the FENE-P relaxation term can be interpreted as a leading Hookean-like contribution plus higher-order corrections associated with finite extensibility. These considerations help rationalize why FENE-P often provides reliable macroscopic stresses in shear-dominated flows over broad parameter ranges, and why it remains a common workhorse in industrial simulations \cite{Yamani23}. Related entropy/energy estimates for Oldroyd-B and associated macro-macro models also support the numerical robustness of closed conformation-tensor
descriptions and guide stable discretizations \cite{Hu07}.
However, the FENE-p approximation fail or produce nonphysical results in some extreme cases, e.g. strong elongation flow.
Therefore, a more physical closure approach, the quasi-equilibrium approximation (QE~\cite{GorbaKZ2004ConstructiveMethods}), is advocated. QE approximation has been adopted and efficiently implemented for both soft-polymer~\cite{Ilg02,Wang08} and rod-like polymers~\cite{chaubal_closure_1998,YuJZ2010,JiangY2021,WeadySS2022ThermoCt,WeadySS2022Fast}. 

\subsection{FENE-QE model and PLA algorithm}
The FENE-QE model \cite{Ilg02} offers a more accurate closure by assuming that, at any instant, the distribution function $f(\bm{q},t)$ relaxes to a state that minimizes the free energy subject to the constraint that its second moment is the current conformation tensor $\bm{C}(t)$. This leads to a distribution of the form
\begin{equation}
    \label{eq:f_qe}
    f_{\mathrm{QE}}(\bm{q}) = \frac{1}{Z(\bm{\lambda})} (1 - |\bm{q}|^2)^{b/2} \exp\left(\bm{\lambda}: \bm{q}\bm{q} \right),
\end{equation}
where $\bm{\lambda}(t)$ is a symmetric tensor of Lagrange multipliers enforcing the constraint $\langle \bm{q}\bm{q} \rangle_{\mathrm{QE}} = \bm{C}(t)$. (Note: $\bm{\lambda}$ is not necessarily traceless, but symmetric).
The unclosed evolution equation \eqref{eq:C_unclosed} can be written in terms of the potential $U$ as
\begin{equation}
    \frac{\mathrm{d}\bm{C}}{\mathrm{d}t} = \bm{K}\cdot\bm{C} + \bm{C}\cdot\bm{K}^\top + \frac{2}{\mathrm{De}}\bm{I} - \frac{2}{\mathrm{De}} \langle \bm{q} (\nabla_{\bm{q}} U)^\top \rangle_{\mathrm{QE}}.\label{eq:C_with_U}
\end{equation}
A key identity, derived from $\int_{\Omega} \nabla_{\bm{q}} \cdot (\bm{q} f_{\mathrm{QE}}) \,\mathrm{d}\bm{q} = 0$, provides the closure relation:
\begin{equation}
    \langle \bm{q} (\nabla_{\bm{q}} U)^\top \rangle_{\mathrm{QE}} = \bm{I} + 2 \bm{C} \cdot \bm{\lambda}.
\end{equation}
Substituting this into \eqref{eq:C_with_U} yields the closed FENE-QE evolution equation:
\begin{equation}
    \frac{\mathrm{d}\bm{C}}{\mathrm{d}t} = \bm{K}\cdot\bm{C} + \bm{C}\cdot\bm{K}^\top - \frac{4}{\mathrm{De}} \bm{C} \cdot \bm{\lambda}.\label{eq:feneqe}
\end{equation}
The corresponding polymer stress tensor is given by $\bm{\tau}_p = \langle \bm{q} (\nabla_{\bm{q}} U)^\top \rangle_{\mathrm{QE}} - \bm{I} = 2\bm{C} \cdot \bm{\lambda}$.

The challenge lies in determining $\bm{\lambda}$ from $\bm{C}$. The mapping $\bm{C} \mapsto \bm{\lambda}$ is implicit and lacks an analytical inverse. The Piecewise Linear Approximation (PLA) method \cite{Wang08} addresses this by pre-computing and interpolating this mapping. Since $\bm{C}$ and $\bm{\lambda}$ are coaxial (share the same eigenvectors), the problem reduces to a 3D mapping between their eigenvalues, $\{c_i\} \leftrightarrow \{\lambda_i\}$.
The PLA algorithm at each time step $t_n$ proceeds as follows:
\begin{enumerate}
    \item \textbf{Diagonalize}: Compute eigenvalues $c_1, c_2, c_3$ and the eigenvector matrix $\bm{Q}$ of $\bm{C}^n$.
    \item \textbf{Interpolate}: Use a pre-computed lookup table $\mathcal{T}$ and multivariate linear interpolation to find the eigenvalues $\lambda_1, \lambda_2, \lambda_3$ corresponding to $c_1, c_2, c_3$.
    \item \textbf{Reconstruct}: Assemble the Lagrange multiplier tensor $\bm{\lambda}^n = \bm{Q} \cdot \mathrm{diag}(\lambda_1, \lambda_2, \lambda_3) \cdot \bm{Q}^\top$.
    \item \textbf{Advance}: Use an ODE solver to advance \eqref{eq:feneqe} to the next time step using $\bm{C}^n$ and $\bm{\lambda}^n$.
\end{enumerate}

\subsection{Neural network-accelerated FENE-QE model (FENE-QE-NN)}
While PLA is effective, it is constrained by large storage requirements and potential interpolation errors, particularly in complex flow regimes. To overcome these limitations, we propose the FENE-QE-NN model, which replaces the lookup table and interpolation with a compact, continuously differentiable neural network. The objective is to construct a high-fidelity mapping $\mathcal{N}_{\theta}: \mathbb{R}^3 \to \mathbb{R}^3$ from the eigenvalues of the conformation tensor, $\bm{c}=(c_1, c_2, c_3)$, to the eigenvalues of the Lagrange multiplier tensor, $\bm{\lambda}=(\lambda_1, \lambda_2, \lambda_3)$.

\paragraph{Dataset and Permutation Invariance}
High-quality training data is crucial. We generate a dataset $\mathcal{D} = \{(\bm{c}^{(i)}, \bm{\lambda}^{(i)})\}_{i=1}^N$ by solving the forward problem (mapping $\bm{\lambda}$ to $\bm{c}$) via high-precision numerical quadrature. A key physical constraint is the permutation invariance of the eigenvalues. To simplify the learning task and reduce redundancy, we enforce a canonical ordering, $c_1 \ge c_2 \ge c_3$, for all training inputs. This forces the network to learn the mapping on a fundamental sub-domain, thereby improving both training efficiency and prediction accuracy. The sampling space is restricted to the physically admissible region $0 < \mathrm{tr}(\bm{C}) < 1$.

\paragraph{Network Architecture}
We employ a standard Multi-Layer Perceptron (MLP) consisting of an input layer (3 neurons), two hidden layers (64 neurons each), and an output layer (3 neurons). Critically, we select the hyperbolic tangent (\texttt{tanh}) as the activation function. Unlike piecewise linear functions such as ReLU, the $C^{\infty}$ smoothness of \texttt{tanh} ensures that the resulting constitutive law is smooth. This property is vital for the stability of numerical simulations, especially when stress derivatives are involved.

\paragraph{Hybrid Optimization Strategy}
To achieve the high accuracy required for scientific computing (i.e., near machine precision), we adopt a two-stage hybrid optimization strategy that combines the strengths of first- and second-order methods:
\begin{enumerate}
    \item \textbf{Phase 1: Global Search (Adam)}. The network is pre-trained using the Adam optimizer with a learning rate of $10^{-3}$ for approximately 1000 epochs. Adam's robustness effectively guides the parameters toward the basin of a global minimum.
    \item \textbf{Phase 2: Local Refinement (L-BFGS)}. After Adam converges, we switch to the L-BFGS optimizer. This quasi-Newton method leverages curvature information to achieve super-linear convergence, refining the solution to a tolerance of $10^{-9}$.
\end{enumerate}
Additionally, both input and output data are standardized using Z-score normalization to improve training stability and convergence speed.

\paragraph{Online Implementation}
Once trained, the lightweight network is embedded into the time-stepping solver. The procedure at each time step mirrors the PLA algorithm but replaces the interpolation step with a network forward pass:
\begin{enumerate}
    \item Decompose $\bm{C}^n$ into its eigenvalues $\bm{c}^n$ and eigenvector matrix $\bm{Q}$.
    \item Sort and normalize $\bm{c}^n$, then feed it into the trained network $\mathcal{N}_{\theta}$ to predict the normalized $\hat{\bm{\lambda}}$.
    \item Denormalize and reconstruct the tensor $\bm{\lambda}^n = \bm{Q} \cdot \mathrm{diag}(\hat{\bm{\lambda}}) \cdot \bm{Q}^\top$.
    \item Advance the ODE \eqref{eq:feneqe} using the computed $\bm{\lambda}^n$.
\end{enumerate}

\section{Numerical results and discussion}

\subsection{Validation of the spectral method}
We first validate the accuracy of the proposed spectral method using the Method of Manufactured Solutions (MMS). The weighted $L_{\omega}^2$ error norm is defined as
\begin{equation}
    \| f \|_{L_{\omega}^2} = \left( \int_{\Sigma} \left(\frac{1-p}{2}\right)^s h^2(p,\theta,\varphi) \frac{\sqrt{1+p}}{4}\sin\theta \,\mathrm{d}p\mathrm{d}\theta\mathrm{d}\varphi \right)^{1/2},
\end{equation}
and the relative error is computed as $E_{\omega} = \| f_{\text{num}} - f_{\text{exact}} \|_{L_{\omega}^2} / \| f_{\text{exact}} \|_{L_{\omega}^2}$.
To facilitate a rigorous error analysis, we construct an exact analytical solution of the form $f_{\text{exact}}(\bm{q},t) = f_{eq}^{\boldsymbol{D}}(\bm{q})\exp(-1/t)$. This requires introducing a corresponding source term $S(\boldsymbol{q},t)$ to the right-hand side of the Fokker--Planck equation:
\begin{equation}
    S(\boldsymbol{q},t) = \frac{1}{t^2}f_{eq}^{\boldsymbol{D}}(\bm{q})\exp\left(-\frac{1}{t}\right).
\end{equation}
The physical parameters are set to $b=12.0$, $s=6.0$, $\mathrm{De}=24.0$, and the flow field is $\boldsymbol{K}=\text{diag}(1,-1,0)$. The convergence results are summarized in Table \ref{tab:spectral_error}.

\begin{table}[htbp]
    \centering
    \caption{Convergence analysis of the Jacobi spectral method. Comparisons are made between the $C^1$-continuous basis (JG1) and the $C^{\infty}$-continuous basis (JGinf).}
    \label{tab:spectral_error}
    \begin{tabular}{ccccc}
        \toprule
        \multirow{2}{*}{\textbf{Resolution}} & \multicolumn{2}{c}{\textbf{JG1 ($C^1$)}} & \multicolumn{2}{c}{\textbf{JGinf ($C^\infty$)}} \\
        \cmidrule(lr){2-3} \cmidrule(lr){4-5}
         & \textbf{DoF} & \textbf{Error} & \textbf{DoF} & \textbf{Error}\\
        \midrule
        $M=N=10$ & 726    & $8.76 \times 10^{-2}$ & 256    & $9.26 \times 10^{-2}$\\
        $M=N=20$ & 4,851  & $1.10 \times 10^{-3}$ & 1,661  & $1.37 \times 10^{-3}$\\
        $M=N=30$ & 14,383 & $3.40 \times 10^{-6}$ & 5,216  & $5.11 \times 10^{-6}$\\
        $M=N=40$ & 33,281 & $3.41 \times 10^{-9}$ & 11,921 & $6.30 \times 10^{-9}$\\
        \bottomrule
    \end{tabular}
\end{table}

The semi-log convergence curves are plotted in Figure \ref{fig:spectral_error}, demonstrating the expected spectral accuracy characteristic of spectral methods.
\begin{figure}[htbp]
    \centering
    \includegraphics[width=0.7\textwidth]{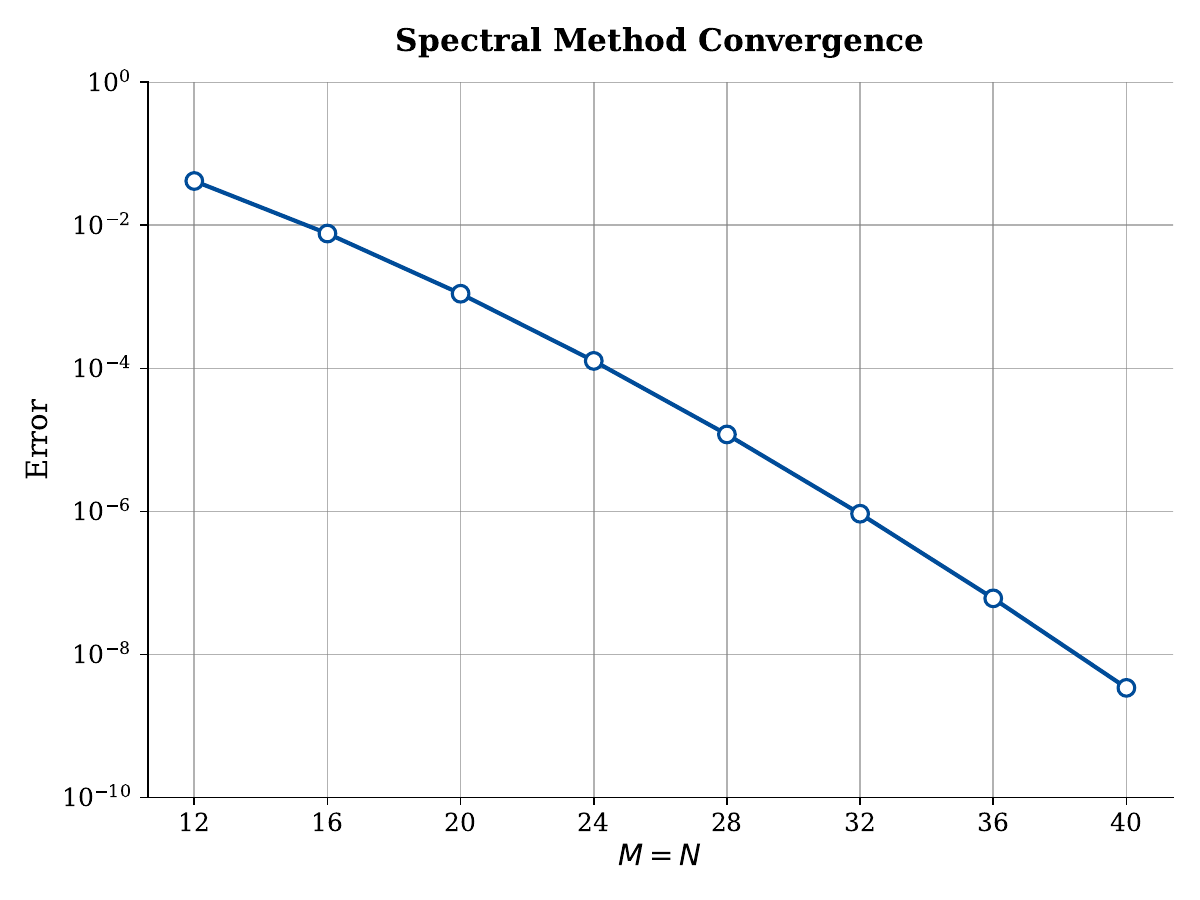}
    \caption{Spectral convergence of the relative $L_{\omega}^2$ error with respect to the polynomial degree $N$.}
    \label{fig:spectral_error}
\end{figure}

\subsection{Training performance of FENE-QE-NN}
The neural network surrogate model is trained using a dataset generated by high-precision numerical quadrature. The network architecture consists of two hidden layers with 64 neurons each. The optimization protocol involves 1000 epochs using the Adam optimizer (global search), followed by L-BFGS (local refinement) with a termination tolerance of $10^{-8}$. The training data is generated with $b=12.0$ and $\mathrm{De}=1.0$.

\begin{figure}[htbp]
    \centering
    \includegraphics[width=1.0\textwidth]{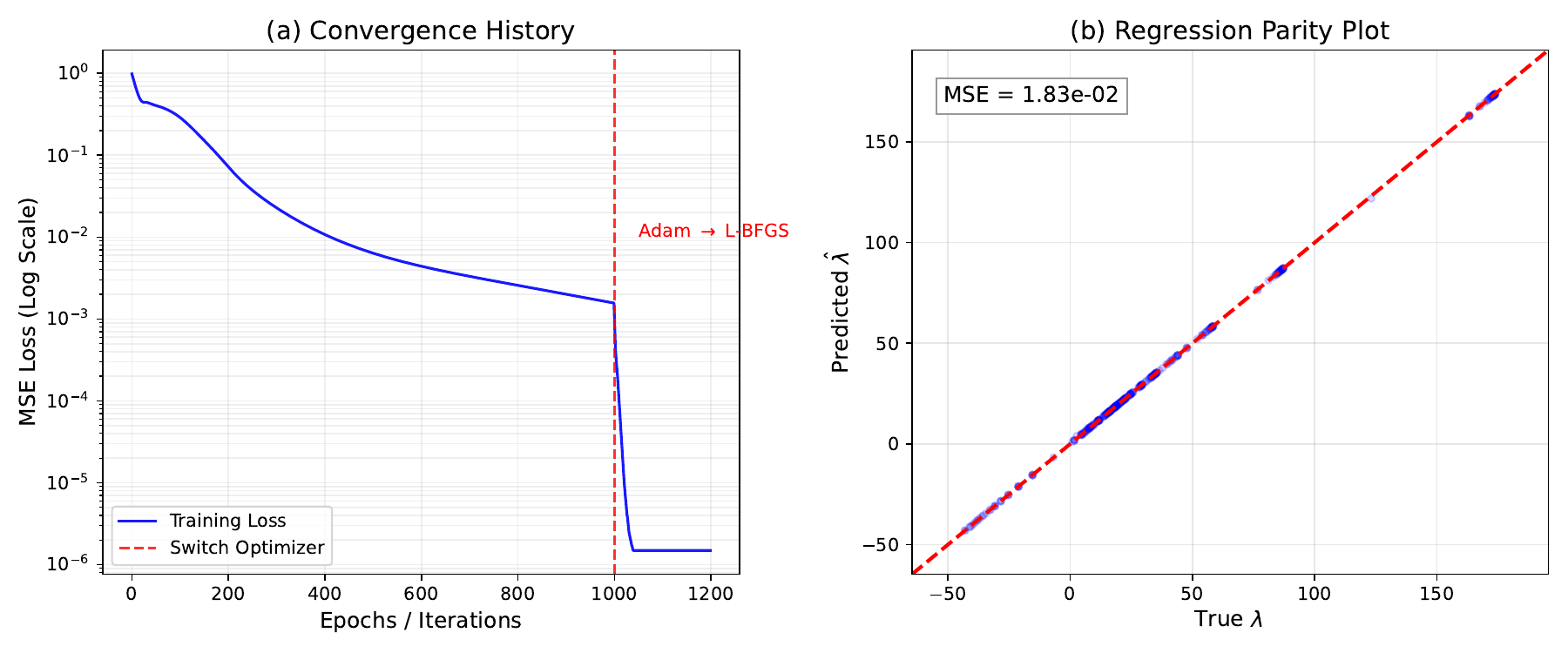}
    \caption{Training diagnostics of the FENE-QE-NN constitutive model. (a) Loss history (MSE) showing the transition from Adam to L-BFGS. (b) Parity plot of the predicted vs. true Lagrange multipliers $\lambda$ on the validation set.}
    \label{fig:nn_train}
\end{figure}

As illustrated in Figure \ref{fig:nn_train}(a), the two-stage optimization strategy proves highly effective. The Adam optimizer rapidly reduces the Mean Squared Error (MSE) to $\mathcal{O}(10^{-3})$, locating the optimal basin of attraction. Subsequently, the L-BFGS optimizer leverages curvature information to achieve super-linear convergence, driving the loss down to $\mathcal{O}(10^{-6})$. Figure \ref{fig:nn_train}(b) assesses the generalization capability. The predictions for the Lagrange multiplier $\lambda$ align perfectly with the ground truth along the $y=x$ line. Notably, even in the high-nonlinearity regime where large $\lambda$ corresponds to the polymer approaching its maximum extensibility ($\text{tr}(\bm{C}) \to 1$), the network maintains high fidelity. The validation MSE is $1.83 \times 10^{-2}$, ensuring sufficient accuracy for the subsequent derivative calculations required for the stress tensor.

\subsection{Benchmark 1: Extensional flow}
We first consider a planar extensional flow defined by the velocity field $\bm{u}=(\kappa x,-\kappa y, 0)$. We employ the high-fidelity spectral solution ($M=N=40$) as the reference to benchmark the accuracy and efficiency of different methods.

\begin{table}[htbp]
    \centering
    \caption{Comparison of computational cost and accuracy for different solvers.}
    \label{tab:time}
    \begin{tabular}{lcc}
        \toprule
        \textbf{Method} & \textbf{CPU Time (s)} & \textbf{$L^2$ Error}\\
        \midrule
        \textbf{Reference ($M=N=40$)} & 568.15 & -- \\
        \textbf{Spectral ($M=N=20$)} & 63.07 & $3.32 \times 10^{-7}$ \\
        \textbf{FENE-P} & 0.002 & $6.77 \times 10^{-2}$ \\
        \textbf{FENE-QE-PLA} & 0.090 & $1.84 \times 10^{-4}$ \\
        \textbf{FENE-QE-NN} & 0.087 & $5.14 \times 10^{-4}$ \\
        \bottomrule
    \end{tabular}
\end{table}

Table \ref{tab:time} highlights the significant speed-up achieved by the closure models. The FENE-QE-PLA and FENE-QE-NN methods are approximately 6500 times faster than the direct high-fidelity spectral solver while maintaining an error level of $\mathcal{O}(10^{-4})$.

Figures \ref{fig:extensional_contours} and \ref{fig:extensional_slices} compare the reconstructed CDFs. At low flow strengths ($\mathrm{De}=1, \kappa=1$), all models agree well. However, as the flow strength increases (e.g., $\mathrm{De}=1, \kappa=20$), the exact CDF develops a distinct bimodal structure. The 1D slices clearly reveal that the FENE-P model, due to its pre-averaging approximation, produces a unimodal Gaussian-like distribution and completely fails to capture the double peaks. In contrast, both FENE-QE-PLA and FENE-QE-NN successfully reproduce the bimodal features, validating the effectiveness of the free-energy minimization assumption in extensional flows.

\begin{figure}[htbp]
    \centering
    \begin{subfigure}[b]{0.63\textwidth}
        \includegraphics[width=\textwidth]{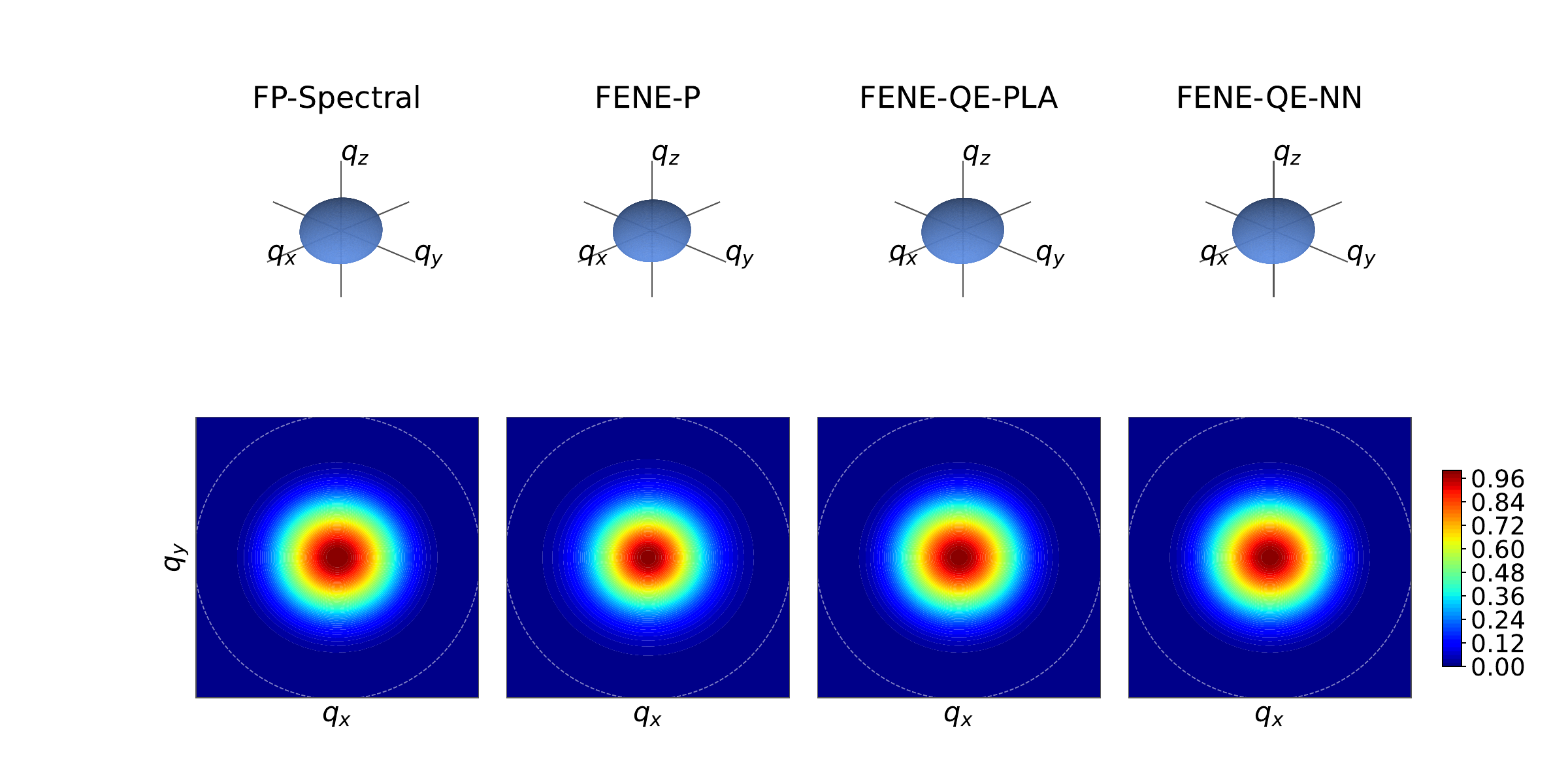}
        \caption{$\mathrm{De}=1.0, \kappa=1.0$}
        \label{fig:k=1}
    \end{subfigure}
    \hfill
    \begin{subfigure}[b]{0.63\textwidth}
        \includegraphics[width=\textwidth]{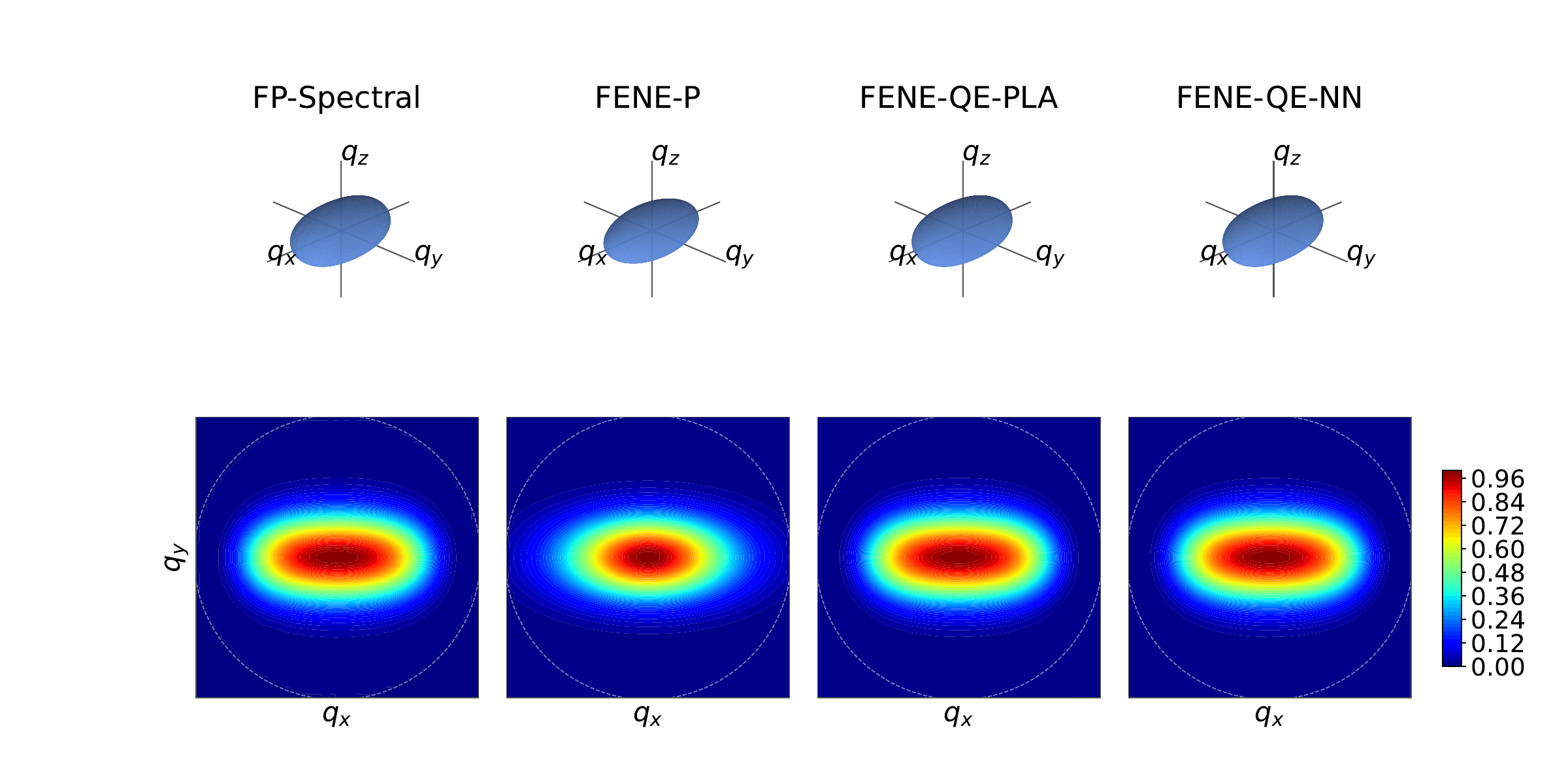}
        \caption{$\mathrm{De}=1.0, \kappa=10.0$}
        \label{fig:k=10}
    \end{subfigure}
    \\
    \begin{subfigure}[b]{0.63\textwidth}
        \includegraphics[width=\textwidth]{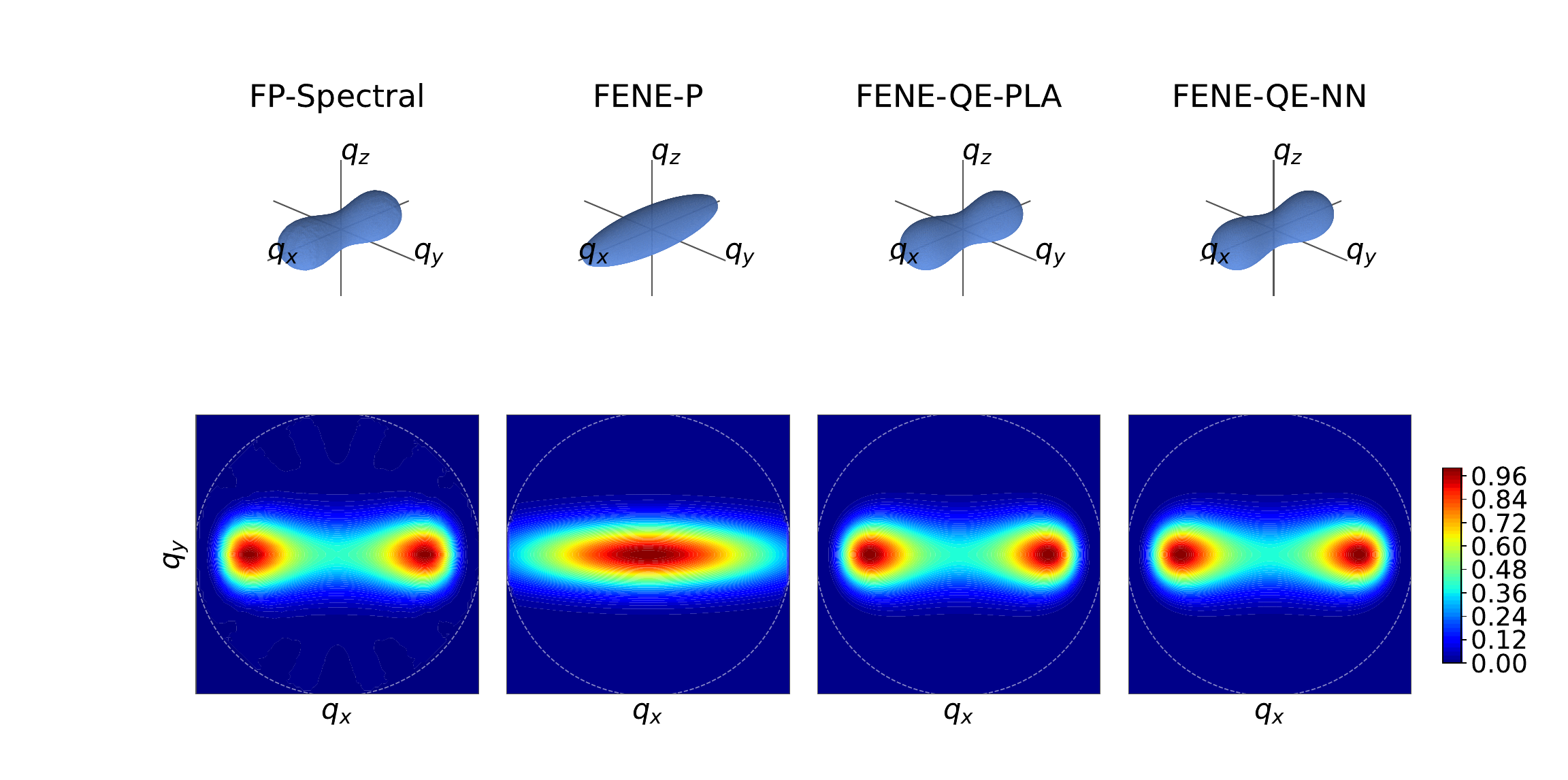}
        \caption{$\mathrm{De}=1.0, \kappa=20.0$}
        \label{fig:k=20}
    \end{subfigure}
    \hfill
    \begin{subfigure}[b]{0.63\textwidth}
        \includegraphics[width=\textwidth]{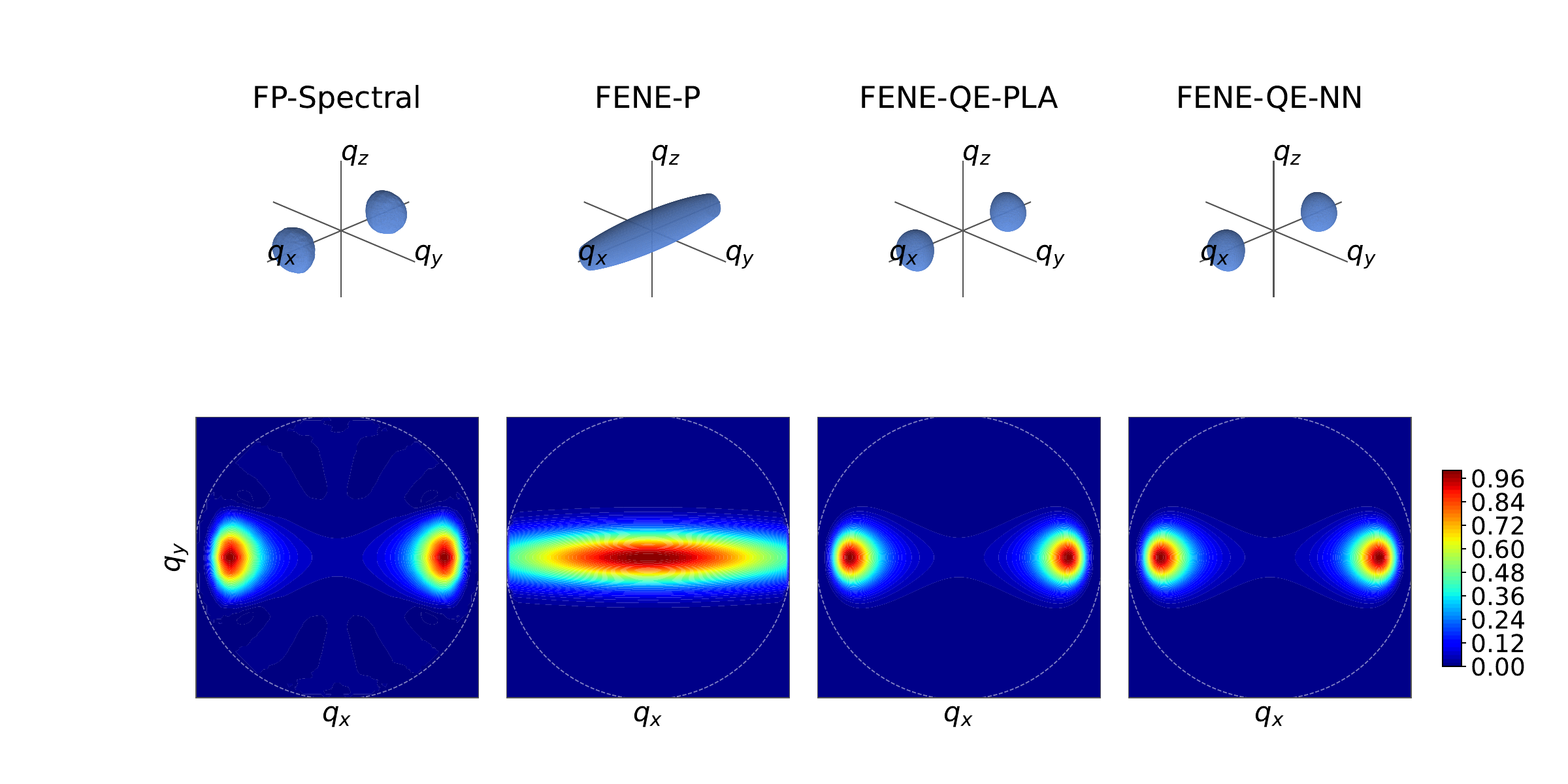}
        \caption{$\mathrm{De}=15.0, \kappa=2.0$}
        \label{fig:de=15}
    \end{subfigure}
    \caption{CDF contours for extensional flow under various conditions.}
    \label{fig:extensional_contours}
\end{figure}

\begin{figure}[htbp]
    \centering
    \begin{subfigure}[b]{0.48\textwidth}
        \includegraphics[width=\textwidth]{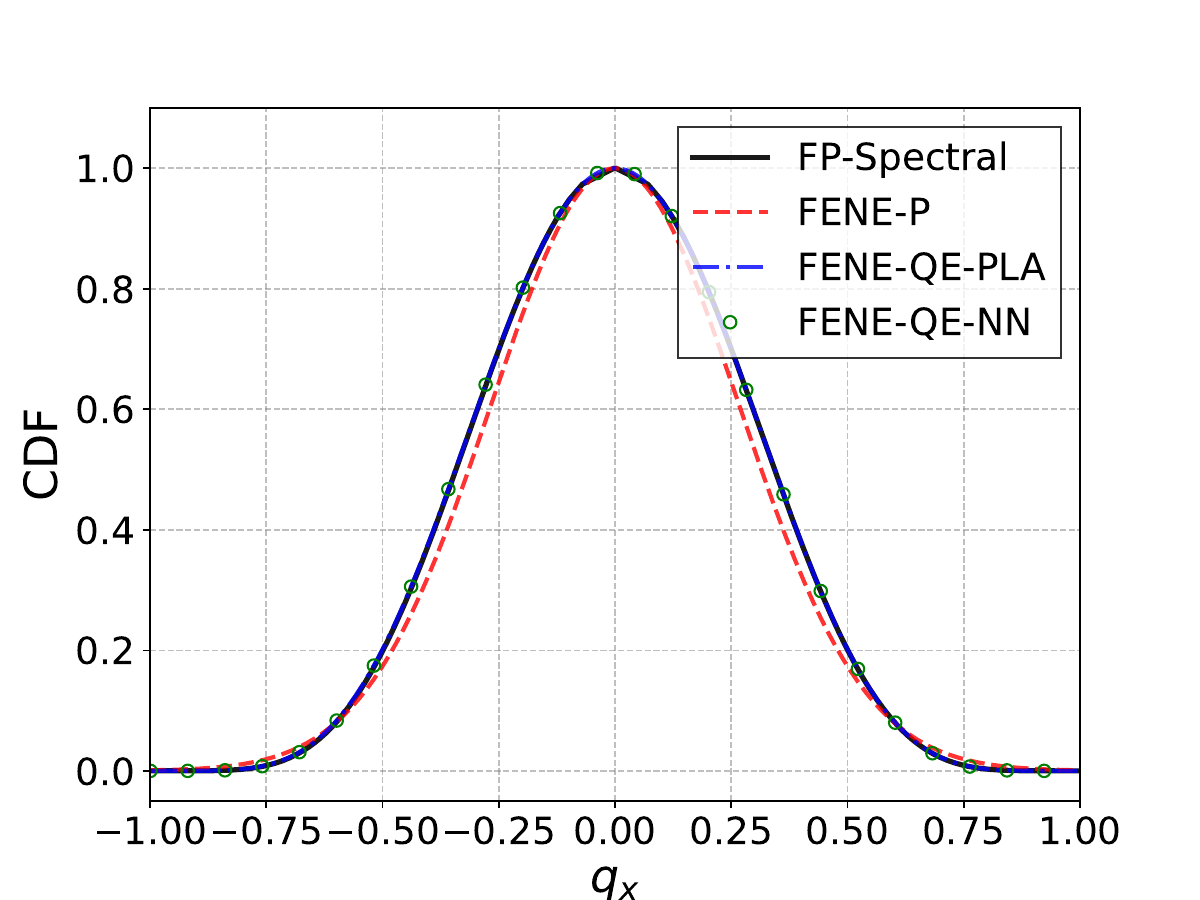}
        \caption{$\mathrm{De}=1, \kappa=1$}
        \label{fig:1Dk=1}
    \end{subfigure}
    \hfill
    \begin{subfigure}[b]{0.48\textwidth}
        \includegraphics[width=\textwidth]{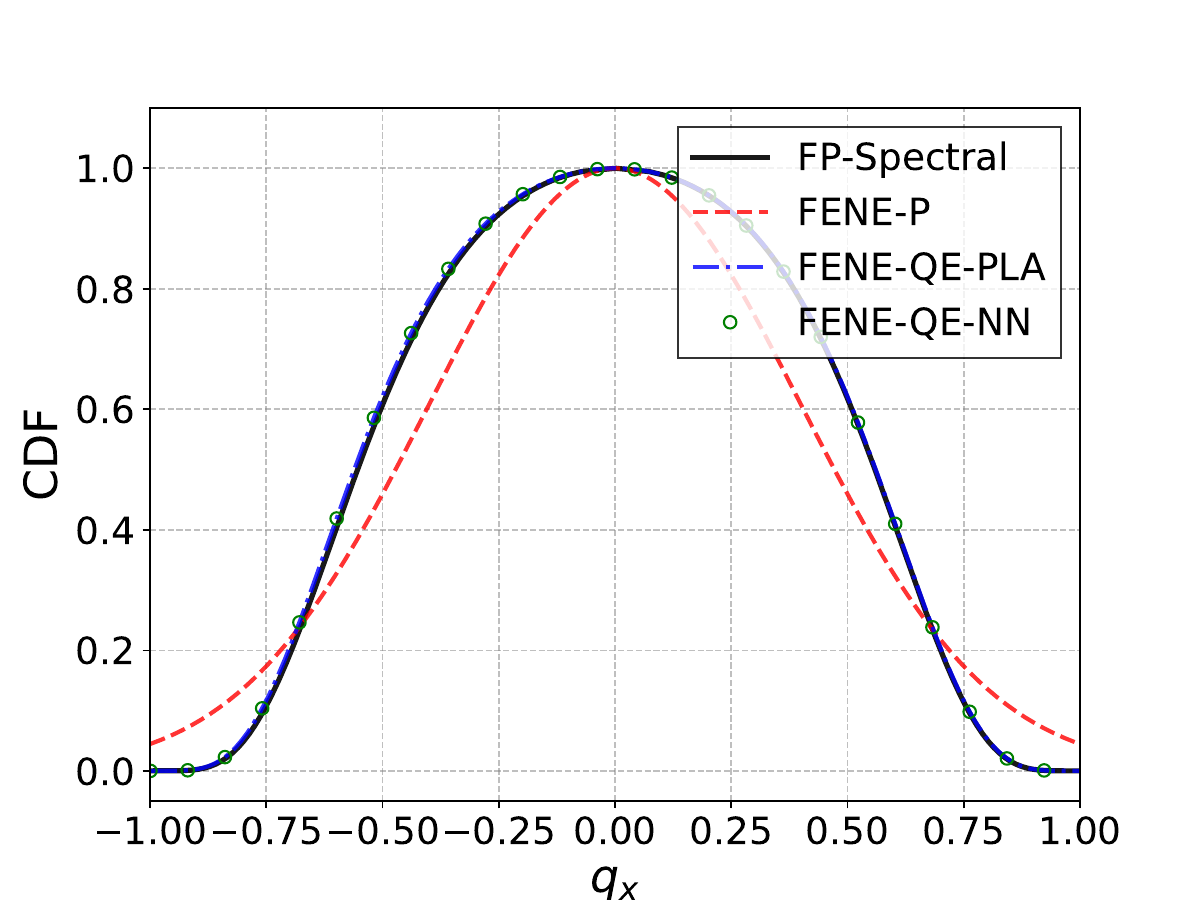}
        \caption{$\mathrm{De}=1, \kappa=10$}
        \label{fig:1Dk=10}
    \end{subfigure}
    \hfill
    \begin{subfigure}[b]{0.48\textwidth}
        \includegraphics[width=\textwidth]{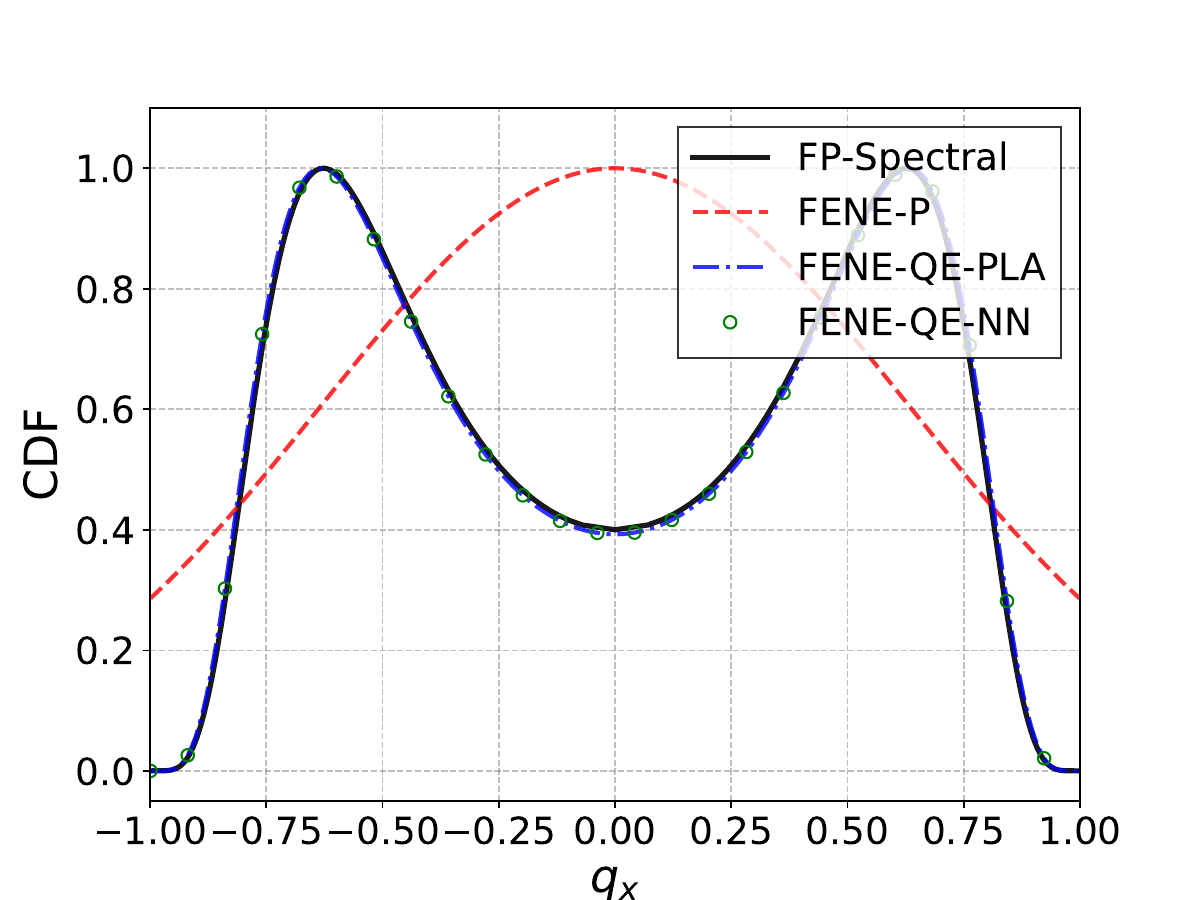}
        \caption{$\mathrm{De}=1, \kappa=20$}
        \label{fig:1Dk=20}
    \end{subfigure}
    \begin{subfigure}[b]{0.48\textwidth}
        \includegraphics[width=\textwidth]{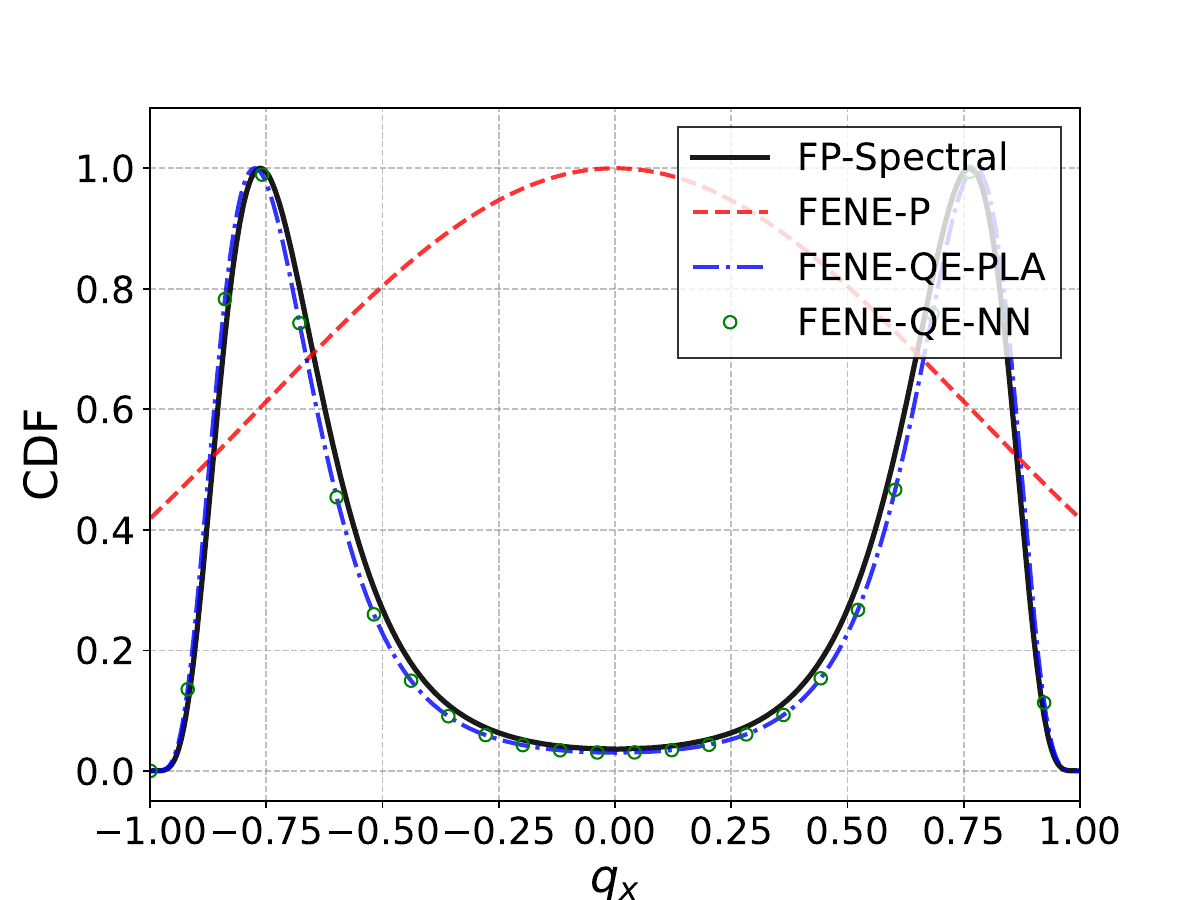}
        \caption{$\mathrm{De}=15, \kappa=2$}
        \label{fig:1Dk=2}
    \end{subfigure}
    \caption{1D cross-sections of the CDF comparing the exact solution with closure models.}
    \label{fig:extensional_slices}
\end{figure}

\subsection{Benchmark 2: Mixed shear and extensional flow and symmetry breaking}
We further introduce a shear component to the flow: $\bm{u}=(x+\kappa y,-y,0)$. The velocity gradient tensor is given by
\begin{equation}
    \bm{K}=
    \begin{pmatrix}
        1 & \kappa & 0 \\
        0 & -1 & 0 \\
        0 & 0 & 0
    \end{pmatrix}.
\end{equation}

\begin{table}[htbp]
    \centering
    \caption{$L^2$ Error of closure models under mixed shear and extensional flow (Fixed $\mathrm{De}=1.0$).}
    \label{tab:error}
    \begin{tabular}{cccc}
        \toprule
        $\kappa$ & \textbf{FENE-P} & \textbf{FENE-QE-PLA} & \textbf{FENE-QE-NN}\\
        \midrule
        1.0 & $6.78 \times 10^{-2}$ & $1.84 \times 10^{-4}$ & $3.36 \times 10^{-4}$\\
        2.0 & $6.81 \times 10^{-2}$ & $4.65 \times 10^{-4}$ & $9.93 \times 10^{-4}$\\
        5.0 & $7.06 \times 10^{-2}$ & $2.18 \times 10^{-3}$ & $2.72 \times 10^{-3}$\\
        10.0 & $7.93 \times 10^{-2}$ & $7.42 \times 10^{-3}$ & $4.68 \times 10^{-3}$\\
        20.0 & $1.12 \times 10^{-1}$ & $3.13 \times 10^{-2}$ & $2.10 \times 10^{-2}$\\
        \bottomrule
    \end{tabular}
\end{table}

\paragraph{Symmetry Breaking}
Table \ref{tab:error} indicates that as the shear rate $\kappa$ increases, the accuracy of the FENE-QE-based models (both PLA and NN) degrades more noticeably compared to the pure extensional case. This phenomenon arises from the intrinsic symmetry limitations of the FENE-QE closure.
The FENE-QE ansatz assumes the distribution function takes the form $f_{\text{QE}} \propto (1-|\bm{q}|^2)^{b/2}\exp(\bm{\lambda}:\bm{q}\bm{q})$. Since $\bm{\lambda}$ is a symmetric tensor, $f_{\text{QE}}$ possesses three mutually orthogonal planes of symmetry defined by the eigenvectors of $\bm{\lambda}$ (or equivalently, $\bm{C}$). However, the exact solution of the Fokker--Planck equation under strong shear flow does not necessarily satisfy this specific symmetry group. As visualized in Figure \ref{fig:shear_results}, the true CDF exhibits a skewness that the quadratic exponential form of FENE-QE cannot fully capture.

\paragraph{Numerical stability and regularization}
We specifically address the numerical stability in the high-Deborah-number limit. We observe that the standard FENE-QE model (via exact integration or PLA) exhibits severe stiffness when the polymer is highly stretched ($\text{tr}(\bm{C}) \to 1$).
Mathematically, this stems from the ill-conditioning of the inverse mapping $\bm{C} \mapsto \bm{\lambda}$. While the forward map is smooth, the inverse map becomes singular near the boundary of the configuration space. As $\text{tr}(\bm{C}) \to 1$, the Lagrange multipliers $\lambda_i \to \infty$, causing the condition number of the Jacobian to diverge:
\begin{equation}
    \kappa(\mathcal{J}) = \left\| \frac{\partial \bm{\lambda}}{\partial \bm{C}} \right\| \left\| \left(\frac{\partial \bm{\lambda}}{\partial \bm{C}}\right)^{-1} \right\| \to \infty.
\end{equation}
Consequently, small perturbations in $\bm{C}$ (e.g., due to time-stepping truncation error) result in massive oscillations in $\bm{\lambda}$.

The FENE-QE-NN model demonstrates superior robustness in this regime. This is attributed to two factors:
\begin{enumerate}
    \item \textbf{Smoothness:} The Tanh activation function enforces $C^{\infty}$ continuity, preventing the $C^0$ kinks found in PLA which can trigger divergence.
    \item \textbf{Implicit regularization:} Neural networks exhibit \textit{spectral bias}, preferentially learning low-frequency functions. In the absence of infinite training data at the singularity, the network tends to smooth out the extreme asymptotic behavior of $\bm{\lambda}$ near the boundary. This acts as an effective regularization, preventing the unphysical oscillations observed in the standard PLA approach.
\end{enumerate}

\begin{figure}[htbp]
    \centering
    \begin{subfigure}[b]{0.48\textwidth}
        \includegraphics[width=\textwidth]{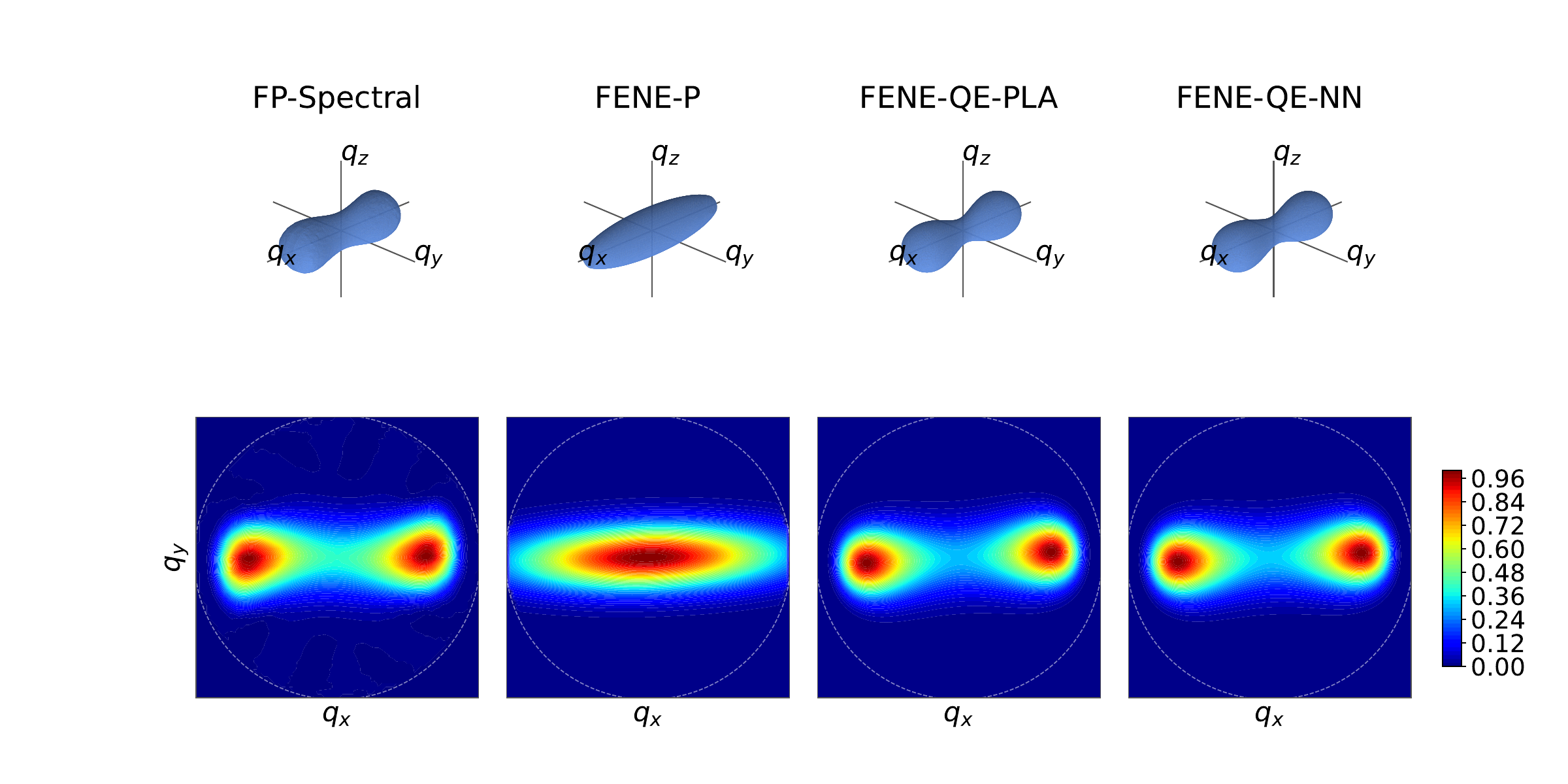}
        \caption{$\mathrm{De}=20.0, \kappa=1.0$}
    \end{subfigure}
    \hfill
    \begin{subfigure}[b]{0.48\textwidth}
        \includegraphics[width=\textwidth]{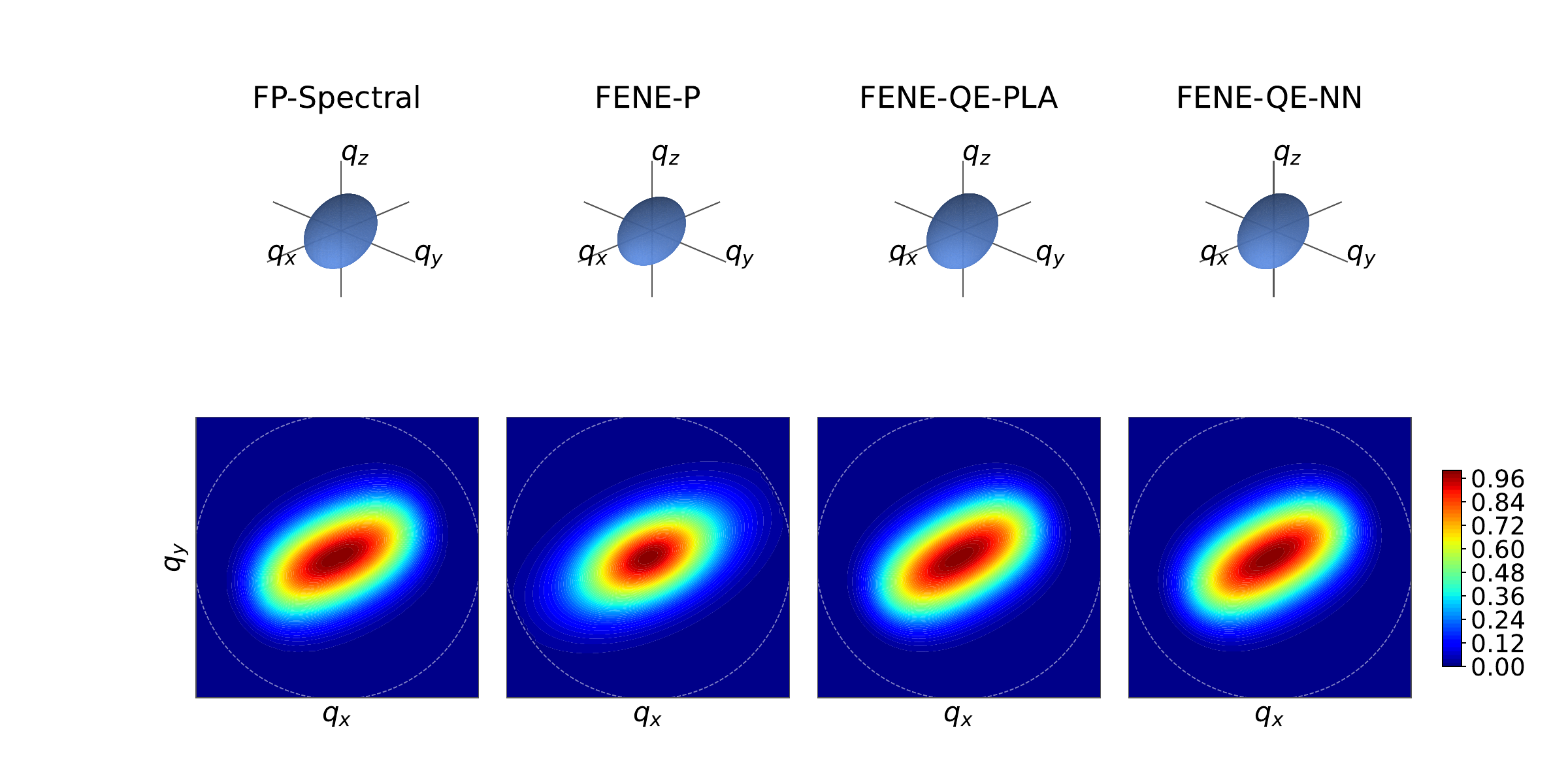}
        \caption{$\mathrm{De}=1.0, \kappa=20.0$}
    \end{subfigure}
    \\
    \begin{subfigure}[b]{0.48\textwidth}
        \includegraphics[width=\textwidth]{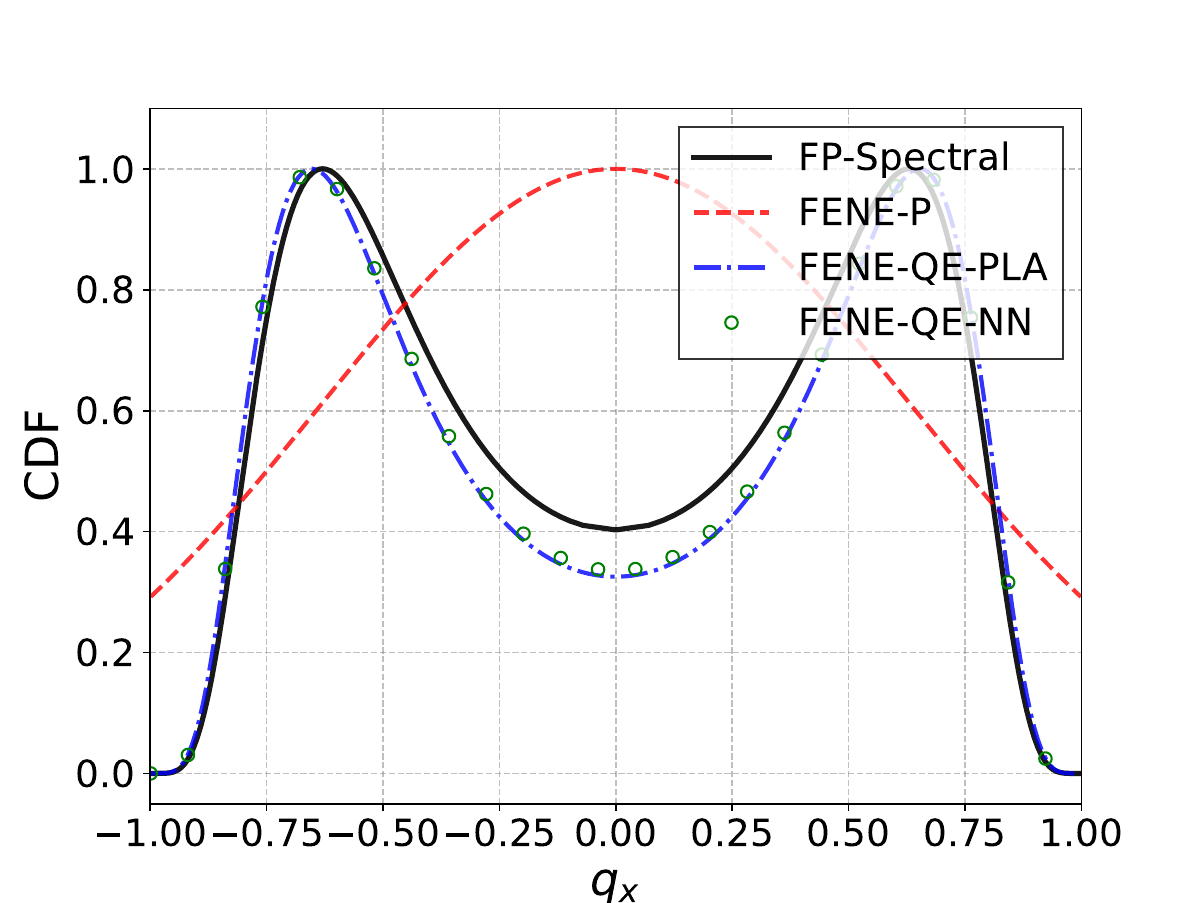}
        \caption{1D Slice(qx): $\mathrm{De}=20, \kappa=1$}
    \end{subfigure}
    \hfill
    \begin{subfigure}[b]{0.48\textwidth}
        \includegraphics[width=\textwidth]{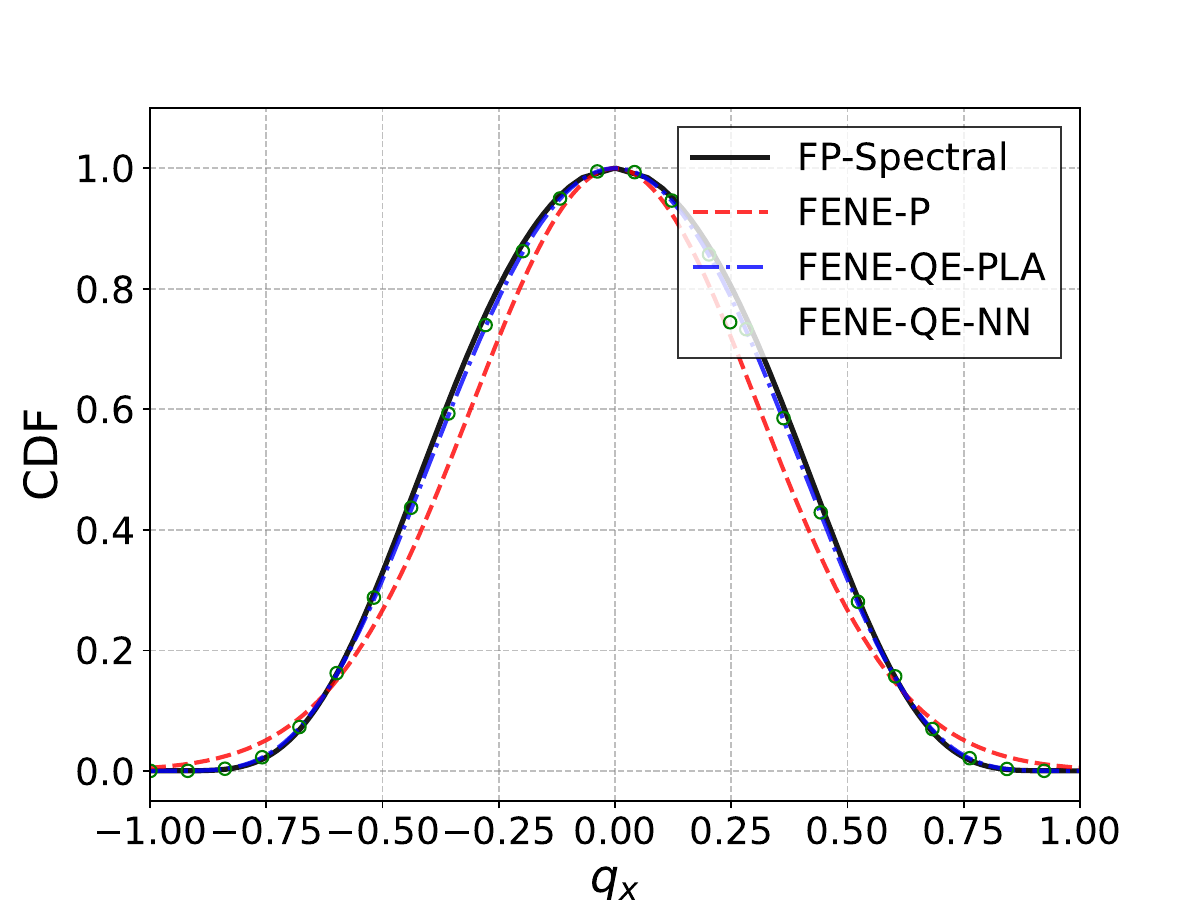}
        \caption{1D Slice(qx): $\mathrm{De}=1, \kappa=20$}
    \end{subfigure}
    \\
    \begin{subfigure}[b]{0.48\textwidth}
        \includegraphics[width=\textwidth]{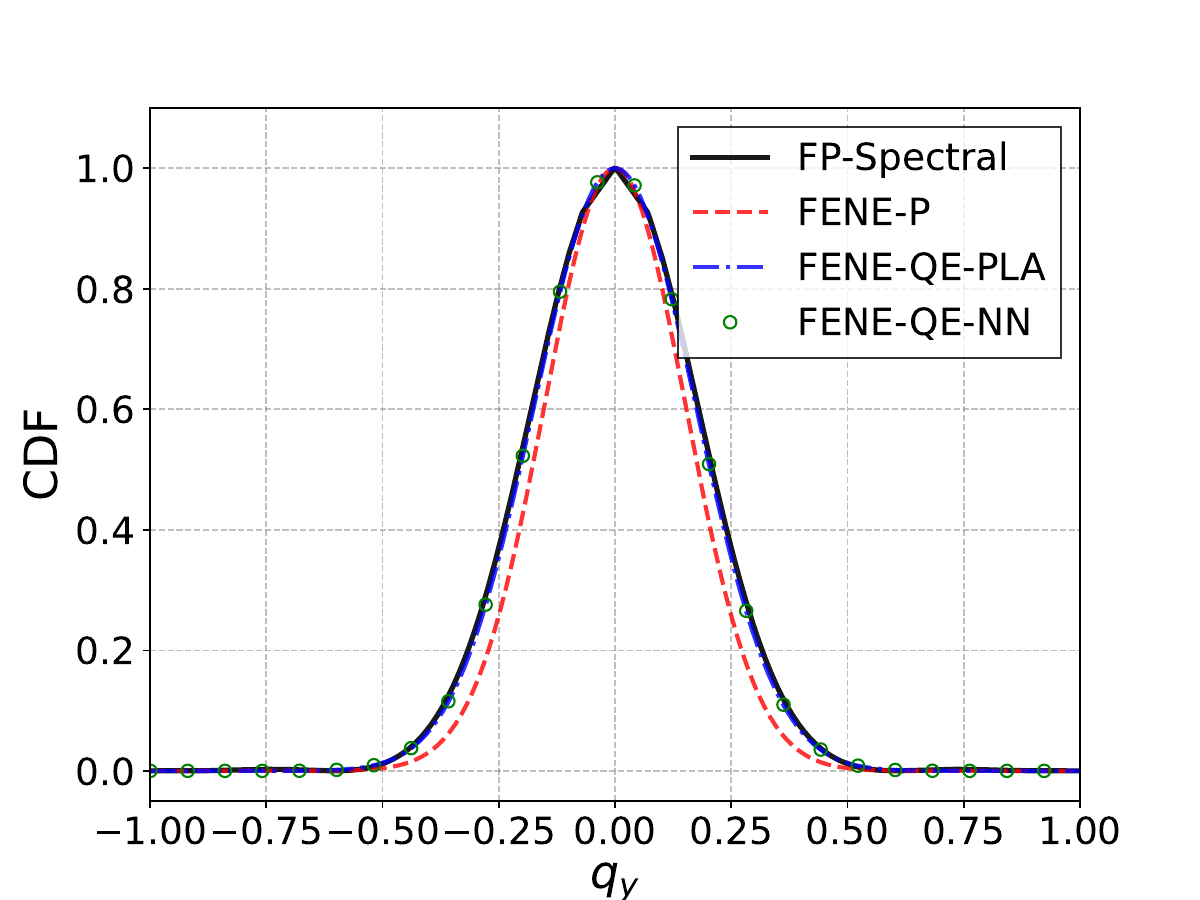}
        \caption{1D Slice(qy): $\mathrm{De}=20, \kappa=1$}
    \end{subfigure}
    \hfill
    \begin{subfigure}[b]{0.48\textwidth}
        \includegraphics[width=\textwidth]{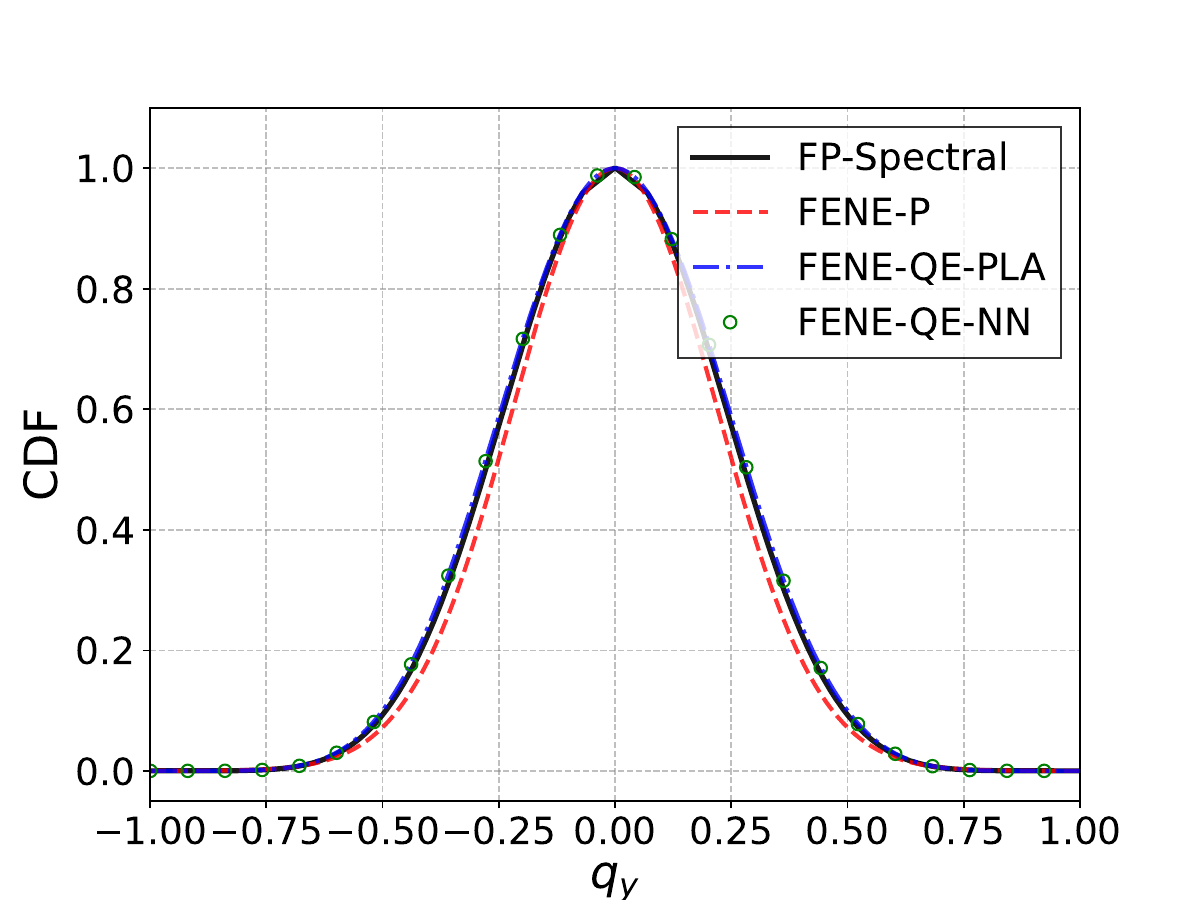}
        \caption{1D Slice(qy): $\mathrm{De}=1, \kappa=20$}
    \end{subfigure}
    \caption{CDF contours and cross-sections for mixed shear and extensional flow.}
    \label{fig:shear_results}
\end{figure}

\subsection{Benchmark 3: Mixed-flow test for FENE-P: microscopic vs.\ macroscopic accuracy}
We assess the performance of the FENE-P closure in a mixed-flow kinematics with the velocity-gradient tensor
\[
\boldsymbol{\kappa}=
\begin{pmatrix}
1 & 1 & 0\\
0 & -1 & 0\\
0 & 0 & 0
\end{pmatrix}.
\]
The errors are evaluated against the reference results, including the CDF \(L^2\) error, the shear stress error in \(\tau_{12}\), and the absolute error of the first normal-stress difference \(N_1\). The results are summarized in Table~\ref{tab:mixedflow_fenep_error}.
Overall, FENE-P exhibits a relatively large microscopic discrepancy: the CDF \(L^2\) error remains above the ``95\%'' accuracy threshold for all tested \(De\). Moreover, the macroscopic error increases markedly in stronger flows (larger \(De\)), especially for \(N_1\) at \(De=10\), highlighting the limitation of the closure in highly elastic regimes. Nevertheless, for moderate elasticity (e.g., \(De=1,2\)), the errors in \(\tau_{12}\) and \(N_1\) are comparatively small and thus remain acceptable for engineering-level predictions where macroscopic quantities are of primary interest.

\begin{table}[t]
\centering
\caption{Errors of FENE-P in mixed flow.}
\label{tab:mixedflow_fenep_error}
\begin{tabular}{c c c c}
\toprule
De & CDF \(L^2\) error & Shear stress error \(\;(\tau_{12})\) & \(N_1\) abs.\ error \\
\midrule
1  & \(6.777661\times10^{-2}\) & \(3.305862\times10^{-3}\) & \(1.661804\times10^{-2}\) \\
2  & \(6.956758\times10^{-2}\) & \(5.581815\times10^{-3}\) & \(3.619042\times10^{-2}\) \\
5  & \(8.304386\times10^{-2}\) & \(8.305585\times10^{-3}\) & \(1.327803\times10^{-1}\) \\
10 & \(1.414949\times10^{-1}\) & \(4.578891\times10^{-3}\) & \(6.120807\times10^{-1}\) \\
\bottomrule
\end{tabular}
\end{table}

\begin{remark}[Numerical Definiteness]
In addition to accuracy limitations, we observed a numerical issue specific to the discrete FENE-P formulation: the computed conformation tensor (second moment) \(\mathbf C\) may lose positive definiteness in some runs. This indicates that standard discretizations of FENE-P do not automatically preserve the SPD property of \(\mathbf C\), and stabilization/positivity-preserving treatments may be required in practice.
\end{remark}

\section{Conclusion}
In this paper, we have established and analyzed a rigorous deterministic framework for simulating dilute polymer solutions within a high-dimensional configuration space. Building upon our previous work, we introduced a fast Spectral-Galerkin method utilizing a basis constructed from mapped Jacobi polynomials and spherical harmonics. This approach effectively resolves the singularity of the FENE probability density function near the maximum extensibility limit while preserving spectral accuracy. We provided a comprehensive mathematical analysis, proving the long-term energy stability of the semi-implicit BDF2 time discretization scheme. Rigorous numerical benchmarks confirm that the proposed algorithm achieves exponential convergence, establishing it as a reliable reference tool for validating macroscopic constitutive models.
Leveraging the high-fidelity data generated by our spectral solver, we conducted a systematic comparative study of closure approximations under complex flow conditions. Our results highlight the limitations of the classic FENE-P model, specifically its inability to capture the bimodal distribution of polymer configurations in strong extensional flows. Despite being less accurate than FENE-QE, the FENE-P model remains widely used in engineering practice due to its simplicity and its reasonably accurate predictions of key macroscopic stresses in weak-to-moderate flow regimes. In contrast, while the FENE-QE model demonstrates superior physical fidelity, we identified critical drawbacks in its standard PLA implementation: high computational cost and severe numerical stiffness arising from the ill-conditioned inverse mapping near the phase space boundary.
To address these challenges, we proposed the FENE-QE-NN model, a data-driven closure that employs a deep neural network to approximate the mapping from the conformation tensor to the Lagrange multipliers. By utilizing smooth activation functions and a hybrid optimization strategy, our model not only retains the thermodynamic consistency and accuracy of the FENE-QE closure but also achieves computational efficiency comparable to the FENE-P model. Furthermore, we demonstrated that the neural network exhibits an intrinsic regularization effect, effectively smoothing the singularity that compromises traditional methods, thereby significantly enhancing numerical robustness in high-Deborah-number regimes.
This work suggests several promising avenues for future research. A natural extension is to apply this framework to non-homogeneous flows, such as 1+3D or 2+3D systems, where spatial variations are significant. In this context, the spectral element method would be an ideal candidate for physical space discretization to maintain high-order accuracy. Secondly, regarding closure approximations, the accuracy of the FENE-QE model could be further improved by incorporating higher-order moment information, such as constraints on fourth-order moments, to better capture non-Gaussian features. Finally, given the demonstrated success of our neural network in solving this challenging inverse problem, we intend to explore more advanced architectures for data-driven model reduction. Our ultimate goal is to develop general-purpose machine learning surrogates that serve as universal constitutive laws for complex fluids, bridging the gap between micro-scale kinetic theory and macro-scale continuum mechanics.

\section*{Acknowledgments}
We would like to acknowledge the financial support from the National Natural Science Foundation of China (grant
no. 12494543, 12171467, W2431008, 12531015) and
the Strategic Priority Research Program of the Chinese Academy of Sciences (grant no. XDA0480504). H. Yu would also like to thank Prof. Nansheng Liu at University of Science and Technology of China for helpful discussions.

\appendix

\section{Formula for Evaluating $D_{ij}$}
\label{app:formula_Dij}

To evaluate $D_{ij}$, we provide subroutines to calculate the angular interaction matrices $U_{ij}$, $V_{ij}$, and $W_{ij}$. We decompose these integrals into polar components ($F$ matrices) and azimuthal components ($E$ matrices).

\subsection{Azimuthal Sub-matrices ($E$)}
The azimuthal integrals involve the basis functions $e_m^\nu(\varphi)$. We define the following sparse sub-matrices:
\begin{align}
    (E_{00})_{\nu m}^{\nu' m'} &= (e_{m}^{\nu}, e_{m'}^{\nu'}) = \delta_{\nu \nu'}\delta_{m m'}, \\
    (E_{01})_{\nu m}^{\nu' m'} &= (e_{m}^{\nu}, e_{m'}^{\nu'} \cos \varphi), \\
    (E_{11})_{\nu m}^{\nu' m'} &= (e_{m}^{\nu}, e_{m'}^{\nu'} \sin \varphi), \\
    (E_{02})_{\nu m}^{\nu' m'} &= (e_{m}^{\nu}, e_{m'}^{\nu'} \cos 2\varphi), \\
    (E_{12})_{\nu m}^{\nu' m'} &= (e_{m}^{\nu}, e_{m'}^{\nu'} \sin 2\varphi).
\end{align}
Additionally, terms involving the derivative $\partial_\varphi$ are defined as:
\begin{align}
    (E_{00}')_{\nu m}^{\nu' m'} &= (e_{m}^{\nu}, \partial_\varphi e_{m'}^{\nu'}), \\
    (E_{01}')_{\nu m}^{\nu' m'} &= (e_{m}^{\nu}, \partial_\varphi e_{m'}^{\nu'} \cos \varphi), \\
    (E_{11}')_{\nu m}^{\nu' m'} &= (e_{m}^{\nu}, \partial_\varphi e_{m'}^{\nu'} \sin \varphi), \\
    (E_{02}')_{\nu m}^{\nu' m'} &= (e_{m}^{\nu}, \partial_\varphi e_{m'}^{\nu'} \cos 2\varphi), \\
    (E_{12}')_{\nu m}^{\nu' m'} &= (e_{m}^{\nu}, \partial_\varphi e_{m'}^{\nu'} \sin 2\varphi).
\end{align}
These matrices are all sparse, with non-zero values only when $m = |m' \pm 1|$ or $m = m', |m' \pm 2|$.

\subsection{Polar Sub-matrices ($F$)}
The polar integrals involve the normalized associated Legendre polynomials $\bar{P}_l^m(\cos\theta)$. We define:
\begin{align}
    (F_{00})_{lm}^{l'm'} &= (\bar{P}_{l}^{m}, \bar{P}_{l'}^{m'})_{\sin\theta}, \\
    (F_{01})_{lm}^{l'm'} &= (\bar{P}_{l}^{m}, \bar{P}_{l'}^{m'} \frac{\cos\theta}{\sin\theta})_{\sin\theta}, \\
    (F_{02})_{lm}^{l'm'} &= (\bar{P}_{l}^{m}, \bar{P}_{l'}^{m'} \cos 2\theta)_{\sin\theta}, \\
    (F_{12})_{lm}^{l'm'} &= (\bar{P}_{l}^{m}, \bar{P}_{l'}^{m'} \sin 2\theta)_{\sin\theta}.
\end{align}
And the derivative terms involving $\partial_\theta$:
\begin{align}
    (F_{00}')_{lm}^{l'm'} &= (\bar{P}_{l}^{m}, \partial_\theta \bar{P}_{l'}^{m'})_{\sin\theta}, \\
    (F_{02}')_{lm}^{l'm'} &= (\bar{P}_{l}^{m}, \partial_\theta \bar{P}_{l'}^{m'} \cos 2\theta)_{\sin\theta}, \\
    (F_{12}')_{lm}^{l'm'} &= (\bar{P}_{l}^{m}, \partial_\theta \bar{P}_{l'}^{m'} \sin 2\theta)_{\sin\theta}.
\end{align}
These matrices depend on $l$ and $l'$ but are independent of $\nu$ and $\nu'$. The values are non-zero only for $|l-l'| \le 2$.

\subsection{Assembly of $U, V, W$}
Using the sub-matrices defined above, the components of the angular matrices are assembled as follows:

\paragraph{1. Matrices $U_{ij}$ (Geometric terms from $r_i r_j$)}
\begin{align*}
    (U_{11}) &= \frac{1}{4}(F_{00}-F_{02})(E_{00}+E_{02}), \\
    (U_{22}) &= \frac{1}{4}(F_{00}-F_{02})(E_{00}-E_{02}), \\
    (U_{33}) &= \frac{1}{2}(F_{00}+F_{02})\delta_{\nu m}^{\nu' m'}, \\
    (U_{12}) &= (U_{21}) = \frac{1}{4}(F_{00}-F_{02})(E_{12}), \\
    (U_{13}) &= (U_{31}) = \frac{1}{2}(F_{12})(E_{01}), \\
    (U_{23}) &= (U_{32}) = \frac{1}{2}(F_{12})(E_{11}).
\end{align*}

\paragraph{2. Matrices $V_{ij}$ (Terms with $\partial_\theta$)}
\begin{align*}
    (V_{11}) &= \frac{1}{4}(F_{12}')(E_{00}+E_{02}), \\
    (V_{22}) &= \frac{1}{4}(F_{12}')(E_{00}-E_{02}), \\
    (V_{33}) &= -\frac{1}{2}(F_{12}')\delta_{\nu m}^{\nu' m'}, \\
    (V_{12}) &= (V_{21}) = \frac{1}{4}(F_{12}')(E_{12}), \\
    (V_{13}) &= \frac{1}{2}(F_{00}'+F_{02}')(E_{01}), \\
    (V_{23}) &= \frac{1}{2}(F_{00}'+F_{02}')(E_{11}), \\
    (V_{31}) &= -\frac{1}{2}(F_{00}'-F_{02}')(E_{01}), \\
    (V_{32}) &= -\frac{1}{2}(F_{00}'-F_{02}')(E_{11}).
\end{align*}

\paragraph{3. Matrices $W_{ij}$ (Terms with $\partial_\varphi$)}
\begin{align*}
    (W_{11}) &= -\frac{1}{2}(F_{00})(E_{12}'), \\
    (W_{22}) &= -W_{11}, \\
    (W_{3j}) &= 0, \\
    (W_{12}) &= -\frac{1}{2}(F_{00})(E_{00}' - E_{02}'), \\
    (W_{21}) &= \frac{1}{2}(F_{00})(E_{00}' + E_{02}'), \\
    (W_{13}) &= -(F_{01})(E_{11}'), \\
    (W_{23}) &= (F_{01})(E_{01}').
\end{align*}

Since all constituent matrices $E$ and $F$ are block-banded with bandwidths of at most 2, the resulting matrices $U, V, W$ are guaranteed to be sparse.

\bibliographystyle{plain}
\bibliography{references}

\end{document}